\documentclass[11pt]{article}
\usepackage[left=1in,right=1in,top=1in,bottom=1in]{geometry}

\usepackage{graphicx}
\usepackage{amsfonts}
\usepackage{url}
\usepackage{amsmath,amssymb}
\usepackage{enumerate}
\usepackage{color}
\usepackage{amsthm,amsmath}
\usepackage{bbm}
\usepackage[version=4]{mhchem}
\usepackage{xspace}
\usepackage{multicol}
\usepackage{subcaption}
\usepackage{rotating}

\usepackage{fancybox}

\usepackage{tikz}
\usetikzlibrary{arrows,shapes,positioning}
\usepackage{tikz-cd}

\usepackage[colorlinks]{hyperref}
\begin{document}

\newcommand{\Sp}{\mathcal{S}}
\newcommand{\C}{\mathcal{C}}
\newcommand{\R}{\mathcal{R}}

\newtheorem{thm}{Theorem}
\newtheorem{pro}[thm]{Proposition}
\newtheorem{lem}[thm]{Lemma}
\newtheorem{cor}[thm]{Corollary}

\theoremstyle{definition}
\newtheorem{dfn}[thm]{Definition}
\newtheorem{exa}[thm]{Example}
\newtheorem{rem}[thm]{Remark}

\newcommand{\tr}{'}

\title{Computing Weakly Reversible Deficiency Zero Network Translations Using Elementary Flux Modes}

\author{Matthew D. Johnston\thanks{Corresponding author: \tt{matthew.johnston@sjsu.edu}} \; and Evan Burton \\ \\  Department of Mathematics and Statistics \\ San Jos\'{e} State University \\ One Washington Square \\ San Jos\'{e}, CA, USA 95192}

\maketitle

\begin{abstract}
We present a computational method for performing structural translation, which has been studied recently in the context of analyzing the steady states and dynamical behavior of mass-action systems derived from biochemical reaction networks. Our procedure involves solving a binary linear programming problem where the decision variables correspond to interactions between the reactions of the original network. We call the resulting network a reaction-to-reaction graph and formalize how such a construction relates to the original reaction network and the structural translation. We demonstrate the efficacy and efficiency of the algorithm by running it on 508 networks from the European Bioinformatics Institutes' BioModels database. We also summarize how this work can be incorporated into recently proposed algorithms for establishing mono and multistationarity in biochemical reaction systems.
\end{abstract}

\section{Introduction}
\label{sec:intro}

A \emph{chemical reaction network} (CRN) is given by a directed graph where the vertices are complexes (i.e. linear combinations of the interacting species) and the edges are reactions (i.e. interactions between species). Under appropriate physical assumptions, such as spatial homogeneity and abundant molecularity, the system is often modeled by an autonomous system of ordinary differential equations in the concentrations of the chemical species. The use of such dynamical models is widespread in systems biology \cite{Alon2007,Ingalls}.

The relationship between the structural properties of a CRN and the dynamical and steady state behavior of the resulting dynamical systems have been studied from a variety of perspectives, including flux balance analysis \cite{Orth2010}, extreme pathway analysis \cite{Wiback2002}, and stoichiometric network analysis \cite{C1,C2}. Recent study has focused on a structural parameter known as the \emph{deficiency}. It is known that, if a mass-action system is weakly reversible and has a deficiency of zero, then it necessarily has complex-balanced steady states (\emph{Deficiency Zero Theorem}, \cite{Feinberg1972,H}). Complex-balancing guarantees uniqueness and stability of steady states for all parameter values and initial conditions, and also affords a simple monomial parametrization of the steady state set \cite{H-J1,C-D-S-S}. 
Further connections between the deficiency and the steady states of mass-action systems have been established \cite{Feinberg1987,Feinberg1988,Feinberg1989,Feinberg1995-1,Feinberg1995-2,D-M,C-D-S-S}.

The study of the deficiency was recently initiated in \emph{generalized chemical reaction networks} (GCRNs) \cite{MR2012,MR2014}. In a GCRN, each vertex in the reaction graph is associated with two potentially distinct complexes, one for the stoichiometry and one for the kinetic rate of the reaction. Surprisingly, for weakly reversible generalized mass-action systems which have a stoichiometric and kinetic-order deficiency of zero, we still obtain a simple monomial parametrization of the steady state set. A process for relating CRNs and GCRNs, called \emph{network translation}, was furthermore established in \cite{J1}. Network translation consists of restructuring a given CRN in such a way that the resulting network (a GCRN) can be used to guarantee dynamical and steady state properties of the original CRN. The process has been utilized to establish connections between chemical reaction network theory \cite{Feinberg1979}, the algebraic study of toric varieties \cite{C-D-S-S,M-D-S-C,D-M2018}, and biochemical reaction modeling \cite{J2,Tonello2017,C-Sh}. 
Recent work has also established a deficiency-based method for constructing rational parametrizations of steady state sets for a broad class of mass-action systems \cite{J-M-P}.

\begin{figure}[t]
\begin{center}
\includegraphics[width=16cm]{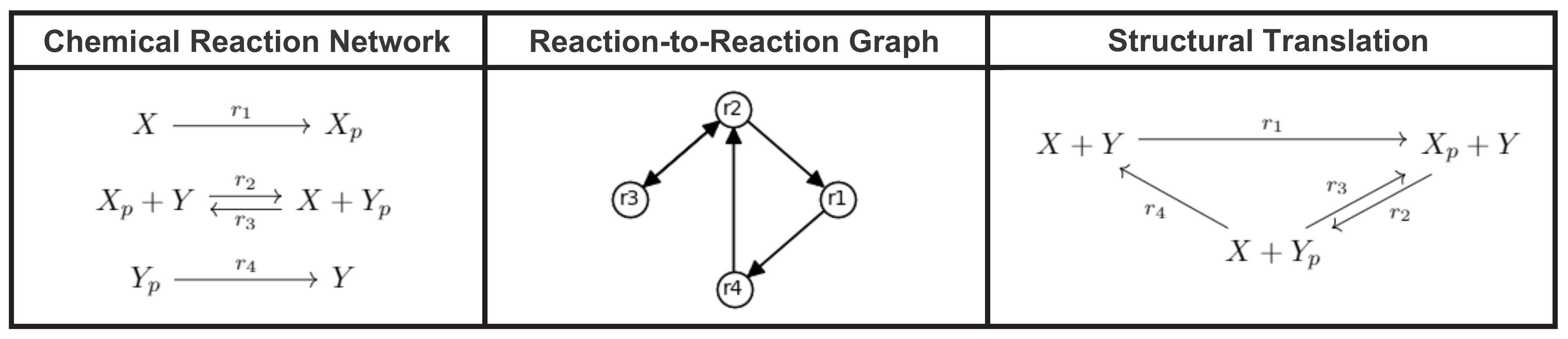}
\caption{A chemical reaction network (left) corresponding to a histidine kinase network where $X$ and $Y$ are two signaling proteins and $p$ is a phosphate group \cite{C-F-M-W2016}. This CRN has elementary flux modes $\{r_1, r_2, r_4\}$ and $\{r_2, r_3\}$ which correspond to the directed cycles in the reaction-to-reaction graph (center). The structural translation (right) has the same elementary flux modes and stoichiometric vectors as the CRN but the elementary flux modes correspond to cycles.}
\label{figure1}
\end{center}
\end{figure}

In this paper, we focus on computational methods for performing the structural component of network translation, which we call \emph{structural translation}. In general, given a biochemical reaction network of realistic scale, it is challenging to determine a suitable (e.g. weakly reversible, deficiency zero) structural translation.  We extend the recent computational work of \cite{J2,Tonello2017} by introducing an elementary flux mode-based approach for performing structural translation. To accomplish this, we introduce a directed graph (called a \emph{reaction-to-reaction graph}) which treats the reactions of a network as vertices and uses the elementary modes to form directed cycles. Under certain rules on the connections on this graph, a weakly reversible and deficiency zero structural translation of the original network can then be constructed. We formulate the construction of this reaction-to-reaction graph as a binary linear programming problem. Such problems can be solved in polynomial time in the number of constraints by Lenstra's algorithm \cite{Lenstra1983}.

Consider the histidine kinase system given in Figure \ref{figure1} (left), which is modified from an example in \cite{C-F-M-W2016} and reproduced in \cite{J-M-P}. 
This network has two elementary flux modes (sets of reactions which balance the net stoichiometry change), namely, $e_1 = \{ r_1, r_2, r_4\}$ and $e_2 = \{ r_2, r_3 \}$. Consistent with these elementary flux modes, we can construct the \emph{reaction-to-reaction} graph given in Figure \ref{figure1} (center) where the reactions are treated as vertices and there is a minimal cycle on each elementary flux mode of the original network. 
From this reaction-to-reaction graph, we can then construct the structural translation of the original network given in Figure \ref{figure1} (right). Notably, the structural translation is weakly reversible and deficiency zero, while the original network is neither.



The paper is organized as follows. In Section \ref{sec:background}, we introduce the terminology and background results relevant to chemical reaction networks and structural translation. In Section \ref{sec:main}, we introduce the notion of a reaction-to-reaction graph, demonstrate how it is related to the structure of a chemical reaction network, and introduce a binary linear programming framework for constructing them. In Section \ref{sec:examples}, we present the output of a run of the algorithm on the European Bioinformatics' BioModels database and detail a few biochemical examples. In Section \ref{sec:conclusions}, we summarize the results of the paper. In Appendix \ref{app:a}, we demonstrate how the results of our algorithm may be utilized to construct steady state parametrizations of mass-action systems according to Lemma 12 and Theorem 14 of \cite{J-M-P} and, when possible, establish mono or multistationarity according to the Corollary 2 of \cite{C-F-M-W2016}.

We use the following notation throughout:
\begin{itemize}
\item
$\mathbb{R}_{> 0}^n = \{ (x_1, \ldots, x_n) \mid x_i > 0, \; i = 1, \ldots, n\}$
\item
$\mathbb{R}_{\geq 0}^n = \{ (x_1, \ldots, x_n) \mid x_i \geq 0, \; i = 1, \ldots, n\}$
\item
$\mathbf{0}^{m \times n}$ is the $m \times n$ matrix with $\mathbf{0}_{i,j} = 0$ for all $i = 1, \ldots, m$ and $j = 1, \ldots, n$
\item
$I^{m \times m}$ is the $m$-dimensional identity matrix
\item
For an indexed set $\mathcal{X} \subseteq \{ X_1, \ldots, X_n \}$, $\mbox{supp}(\mathcal{X}) = \{ i \in \{1, \ldots, n \} \mid X_i \in X\}$.
\item
For a vector $\mathbf{v} \in \mathbb{R}_{\geq 0}^n$, $\mbox{supp}(\mathbf{v}) = \{ i \in \{ 1, \ldots, n \} \mid v_i > 0 \}$
\end{itemize}

\section{Background}
\label{sec:background}

In this section, we present the terminology relevant to chemical reaction networks, structural translation, and elementary flux modes. Note that we only introduce the terminology required to establish the computational program presented in Section \ref{sec:blp}. In particular, we do not use the full generality of generalized chemical reaction networks as given in \cite{MR2014,J-M-P}.

\subsection{Chemical Reaction Networks}
\label{sec:crn}

We define a \emph{species set} $\Sp = \{ X_1, \ldots, X_m\}$ and a \emph{complex set} $\C = \{ y_1, \ldots, y_n\}$ whose elements (complexes) are linear combinations of the species, i.e.
\[y_i = \sum_{j=1}^m y_{ij} X_j, \hspace{0.2in} i=1,\ldots, n.\]
The coefficients $y_{ij} \in \mathbb{Z}_{\geq 0}$ are called \emph{stoichiometric coefficients}. Allowing a slight abuse of notation, we let $y_i$ denote both the complex itself and the corresponding \emph{complex vector} $y_i = (y_{i1}, \ldots, y_{im}) \in \mathbb{Z}_{\geq 0}^m$. The \emph{reaction set} is given by $\R = \{ r_1, \ldots, r_r \} \subseteq \C \times \C$ where we represent individual reactions as either ordered pairs of complexes (i.e. $r_k = (y_i,y_j)$) or directed edges (i.e. $r_k = y_i \to y_j$). It will occasionally be convenient to use mappings $s, p: \mbox{supp}(\R) \mapsto \mbox{supp}(\C)$ such that $s$ (respectively $p$) maps the source (respectively product) of each reaction to the corresponding complex, i.e. $r_k = y_{s(k)} \to y_{p(k)}$. A \emph{chemical reaction network} (CRN) is given by the triple $(\Sp, \C, \R)$.



The \emph{reaction graph} of a CRN is the directed graph $G = (V,E)$ where the vertices are the complexes (i.e. $V = \C$) and the edges are the reactions (i.e. $E = \R$). 
The connected components of the reaction graph of a CRN are called \emph{linkage classes} while the strongly connected components are called \emph{strong linkage classes}. We will let $\ell$ denote the number of linkage classes in a CRN. A CRN is said to be \emph{weakly reversible} if its linkage classes and strong linkage classes coincide. To each reaction $r_k = y_{s(k)} \to y_{p(k)} \in \mathcal{R}$ we may associate a \emph{reaction vector} $y_{p(k)} - y_{s(k)} \in \mathbb{Z}^{m}$. The \emph{stoichiometric matrix} of a CRN is given by the matrix $\Gamma \in \mathbb{Z}^{m \times n}$ with columns defined by $\Gamma_{\cdot, k} = y_{p(k)} - y_{s(k)}$. The \emph{stoichiometric subspace} of a CRN is given by $S = \mbox{im}(\Gamma)$.

Consider a time-dependent vector of chemical concentrations $\mathbf{x}(t) = (x_1(t), \ldots, x_m(t)) \in \mathbb{R}_{\geq 0}^m$. Assuming sufficient molecularity of chemical species and mass-action kinetics, it is common to assign each reaction $r_i \in \R$ a rate constant $k_i \in \mathbb{R}_{> 0}$ and model the evolution of $\mathbf{x}(t)$ via the \emph{mass-action system}
\begin{equation}
\label{eqn:de}
\frac{d\mathbf{x}}{dt} = \Gamma R(\mathbf{x})
\end{equation}
where $R(\mathbf{x}) \in \mathbb{R}_{\geq 0}^r$ has entries $R_i(\mathbf{x}) = \prod_{j=1}^m x_j^{[y_{s(i)}]_j}$ \cite{G-W1}. Other widely-used kinetic choices for $R(\mathbf{x})$ include Michaelis-Menten and Hill kinetics \cite{M-M,Hi}. Note that $d\mathbf{x}/dt \in S$ for all $t \geq 0$ and consequently solutions are restricted to \emph{stoichiometric compatibility classes}, i.e. $\mathbf{x}(t) \in (S+ \mathbf{x}_0) \cap \mathbb{R}_{\geq 0}^m$ for $t \geq 0$. The analysis we perform in this paper will focus largely on the structural aspects of CRNs rather than the dynamical equations \eqref{eqn:de}. That is, we focus on $\Gamma$ rather than $R(\mathbf{x})$.

We may further factor the stoichiometric matrix $\Gamma$ by introducing a \emph{complex matrix} $Y \in \mathbb{Z}_{\geq 0}^{m \times n}$ with columns $Y_{\cdot,i} = y_i$ and an \emph{incidence matrix} $I_a \in \{ -1,0,1 \}^{n \times r}$ with entries $[I_a]_{ik} = -1$ if $s(k) = i$, $[I_a]_{ik} = 1$ if $p(k) = i$, and $[I_a]_{ik} = 0$ otherwise. It can be easily verified that $\Gamma = YI_a$. The \emph{deficiency} of a CRN is a nonnegative parameter defined by $\delta = \mbox{dim}(\mbox{ker} (Y) \cap \mbox{im} (I_a))$. Alternatively, the deficiency can be computed by the formula $\delta = n - \ell - \mbox{dim}(S)$ (see \cite{J1}). The deficiency was first introduced in \cite{Feinberg1972,H} and has been used extensively since in the context of steady states of mass-action systems \cite{H-J1,Feinberg1987,Feinberg1989,Feinberg1995-1,Feinberg1995-2,J1,MR2012}.



Consider the following example.

\begin{exa}
\label{example1}
Reconsider the histidine kinase network from Figure \ref{figure1} (left). 
We have the following sets:
 \begin{equation}
 \label{set1}
 \begin{aligned}
 \Sp & = \{ X, X_p, Y, Y_p \}\\
 \C & = \{ X, X_p, X_p + Y, X+Y_p, Y_p, Y \}\\
 \R & = \{ X \to X_p, X_p + Y \to X + Y_p, X + Y_p \to X_p + Y, Y_p \to Y\}.
 \end{aligned}
 \end{equation}
The network has six complexes ($n = 6$) and three linkage classes ($\ell = 3$). The second linkage class is strongly connected while the first and third are not. It follows that the network is not weakly reversible. Using the ordering of species and reactions given above, we can compute that the network has the following structural matrices
\begin{equation}
\label{gamma1}
\Gamma = \left[ \begin{array}{cccc} -1 & 1 & -1 & 0 \\ 1 & -1 & 1 & 0 \\ 0 & -1 & 1 & 1 \\ 0 & 1 & -1 & -1 \end{array} \right] = \left[ \begin{array}{cccccc} 1 & 0 & 0 & 1 & 0 & 0 \\ 0 & 1 & 1 & 0 & 0 & 0 \\ 0 & 0 & 1 & 0 & 0 & 1 \\ 0 & 0 & 0 & 1 & 1 & 0 \end{array} \right] \left[ \begin{array}{cccc} -1 & 0 & 0 & 0 \\ 1 & 0 & 0 & 0 \\ 0 & -1 & 1 & 0 \\ 0 & 1 & -1 & 0 \\ 0 & 0 & 0 & -1 \\ 0 & 0 & 0 & 1 \end{array} \right] = Y I_a.
\end{equation}
We have that $\mbox{dim}(S) = 2$ so that the deficiency is $\delta = n - \ell - \mbox{dim}(S) = 6 - 3 - 2 = 1$. Alternatively, we can compute that $\delta = \mbox{dim}(\mbox{ker} (Y) \cap \mbox{im} (I_a)) = 1$.
\end{exa}

\subsection{Structural Translation}
\label{sec:translation}

We introduce the following structural notion of network translation, which is weaker than those presented in \cite{J1,J2,Tonello2017,J-M-P}. 

\begin{dfn}
\label{def:translation}
Consider two CRNs $(\Sp,\C,\R)$ and $(\Sp,\C',\R')$ with corresponding complex, incidence, and stoichiometric matrices $Y$, $I_a$, $\Gamma$, $Y'$, $I_a'$, and $\Gamma'$ as defined in Section \ref{sec:crn}. We say that $(\Sp,\C,\R)$ and $(\Sp,\C',\R')$ are \emph{structural translations} of one another if $\Gamma = YI_a = Y'I_a' = \Gamma'$.
\end{dfn}

\noindent Intuitively, two CRNs are structural translations of one another if, despite potentially different complexes and reactions (i.e. the $Y$ and $I_a$), they have the same reaction vectors (i.e. columns of $\Gamma$). In practice, we will typically have a CRN $(\Sp,\C,\R)$ given to us and want to construct a CRN $(\Sp,\C',\R')$ which has specific structure properties. Consequently, we will typically refer to $(\Sp,\C,\R)$ as the original network and $(\Sp,\C',\R')$ as the structural translation.

Network translation can be visualized by the operation of adding or subtracting linear combinations of species, known as \emph{translation complexes}, from individual reactions. 
The summation of the original network's complexes and corresponding translation complexes then produces the translated network's complexes. Formally, we let $\Lambda = \{ \alpha_1, \ldots, \alpha_r \}$ where $\alpha_i \in \mathbb{R}^m$, $i = 1, \ldots, r$, denote a set of \emph{translation complexes}. We represent the operation of translating the reaction $r_k = y_{s(k)} \to y_{p(k)} \in \R$ by the translation complex $\alpha_k$ as
\[\begin{tikzcd}
y_{s(k)} \arrow[r, "r_k"] & y_{p(k)} & (+\alpha_k)
\end{tikzcd}\]
for $k = 1, \ldots, r$. This operation produces the translated reactions $y_{s(k)} + \alpha_k \to y_{p(k)} + \alpha_k \in \mathcal{R}'$ and translated complexes $y_{s(k)} + \alpha_k, y_{p(k)} + \alpha_k \in \C'$. Note that this may produce repeated complexes and therefore new connections in the corresponding reaction graph. Since the net stoichiometric change across each reaction is unaltered by this operation (i.e. $y_{p(k)} - y_{s(k)} = (y_{p(k)} + \alpha_k) - (y_{s(k)} + \alpha_k) = y'_{p(k)} - y'_{s(k)}$) we have that $\Gamma = \Gamma'$ and the networks are structural translations of one another.



Consider the following example.




\begin{exa}
\label{example3}
Reconsider the histidine kinase network from Figure 1 (left) 
taken with the following translation scheme
\begin{equation} \label{histidine-translation}
\begin{tikzcd}
X \arrow[r, "r_1"] & X_p & & (+Y)\\[-0.15in]
X_p + Y \arrow[r,  yshift=+0.5ex,"r_2"] & X + Y_p \arrow[l,  yshift=-0.5ex,"r_3"]  & & (+\emptyset)\\[-0.15in]  
Y_p \arrow[r, "r_4"] & Y & & (+X)
\end{tikzcd}
\end{equation}
That is, we translate $r_1$ by $\alpha_1 = Y$, translate $r_2$ and $r_3$ by $\alpha_2 = \alpha_3 = \emptyset$, and translate $r_4$ by $\alpha_4 = X$. This produces the structural translation in Figure \ref{figure1} (right).  
Notably, the stoichiometric changes across each reaction in the two networks are identical. Formally, 
for the network in Figure \ref{figure1} (right), we have the sets
\[\begin{aligned}
\Sp & = \{ X, X_p, Y, Y_p \} \\
\C & = \{ X+Y, X_p + Y, X+Y_p\}.
\end{aligned}\]
Using this ordering of species and complexes, we can determine the following structural matrices:
\begin{equation}
\label{gamma2}
\Gamma'  = \left[ \begin{array}{cccc} -1 & 1 & -1 & 0 \\ 1 & -1 & 1 & 0 \\ 0 & -1 & 1 & 1 \\ 0 & 1 & -1 & -1 \end{array} \right] = \left[ \begin{array}{ccc} 1 & 0 & 1 \\ 0 & 1 & 0 \\ 1 & 1 & 0 \\ 0 & 0 & 1 \end{array} \right] \left[ \begin{array}{cccc} -1 & 0 & 0 & 1 \\ 1 & -1 & 1 & 0 \\ 0 & 1 & -1 & -1 \end{array} \right]  = Y' I_a'
\end{equation}
Since $\Gamma = \Gamma'$ where $\Gamma$ is from \eqref{gamma1}, we have that the networks in Figure \ref{figure1} (left) and (right) are structural translations of one another by Definition \ref{def:translation}.

It is notable that the CRN in Figure \ref{figure1} (right) is weakly reversible and deficiency zero, while the original CRN in Figure \ref{figure1} (left) is not weakly reversible and has a deficiency of one. The structure of the CRN in Figure \ref{figure1} (right) can be used to establish properties about the steady state set of the mass-action system \eqref{eqn:de} corresponding to the network in Figure \ref{figure1} (left) \cite{J1,J-M-P,M-D-S-C,C-F-M-W2016}. We outline some of these methods in Appendix \ref{app:a}.
\end{exa}

 \subsection{Elementary Flux Modes}
 \label{sec:em}

The following structural property of CRNs will factor significantly in our construction of structural translations in Section \ref{sec:main}.

\begin{dfn}
\label{def:em}
Consider a CRN $(\Sp,\C,\R)$ with stoichiometric matrix $\Gamma$ and incidence matrix $I_a$. Then:
\begin{enumerate}
\item
A vector $e_i \in \mathbb{R}_{\geq 0}^r$ is an \emph{elementary flux mode} of the CRN if $e_i \in \mbox{ker}(\Gamma)$ and $e_i$ is not a convex combination of any other $e_j, e_k \in \mbox{ker}(\Gamma) \cap \mathbb{R}_{\geq 0}^r$. 
The set of elementary flux modes of a CRN will be denoted $\mathcal{E} = \{e_1, \ldots, e_p \}$.
\item
The \emph{elementary flux cone} is defined as $\mbox{cone}(\mathcal{E}) = \mbox{ker}(\Gamma) \cap \mathbb{R}_{\geq 0}^r$.
\item
An elementary flux mode $e_i \in \mathcal{E}$ is called a \emph{cyclic generator} of the CRN if $e_i \in \mbox{ker}(I_a)$.
\item
An elementary flux mode $e_i \in \mathcal{E}$ is called a \emph{stoichiometric generator} of the CRN if $e_i \not\in \mbox{ker}(I_a)$.
\item
The set of elementary flux modes $\mathcal{E}$ is \emph{unitary} if every entry of every $e_i \in \mathcal{E}$ is a one or a zero.
\item
The set of elementary flux modes $\mathcal{E}$ \emph{covers} the reaction set $\R$ if $\mbox{cone}(\mathcal{E}) \cap \mathbb{R}_{> 0}^r \not= \emptyset$.
\end{enumerate}
\end{dfn}

\noindent Note that the set of elementary modes $\mathcal{E} = \{ e_1, \ldots, e_p \}$ consists of the extremal generators of the elementary flux cone, $\mbox{cone}(\mathcal{E})$. 

In this paper, we consider only unitary elementary flux modes. In such cases, we have that $e_i$ is completely determined by $\mbox{supp}(e_i)$ and, consequently, we will allow $e_i$ to correspond to both the elementary flux mode and its support, e.g. we will use $e_i = (1, 0, 1, 1, \ldots)$ and $e_i = \{ r_1, r_3, r_4, \ldots \}$ interchangeable. We may interpret unitary elementary flux modes as sets of reactions which, if taken in any order, would result in no net gain or loss of any species. A cyclic generator furthermore has the property that this sequence of reactions corresponds to a directed cycle in the reaction graph of the CRN. 
Elementary flux modes have played a significant role recently in metabolic engineering, although efficient computation of the set $\mathcal{E}$ remains challenging \cite{Zang2013}. 


Consider the following example.

\begin{exa}
Reconsider the histidine kinase example given in Figure \ref{figure1} (left), and the structural translation given in Figure \ref{figure1} (right). Also consider the corresponding matrices $Y$ and $I_a$ given in \eqref{gamma1} and $Y'$ and $I_a'$ given in \eqref{gamma2}. Since $\Gamma = \Gamma'$, we have that the elementary modes of the two CRNs coincide. We can compute that $e_1 = (1,1,0,1)$ and $e_2 = (0,1,1,0)$. Since $e_1$ and $e_2$ only consist of zeros and ones, we have that the CRNs have unitary elementary modes. We therefore write $e_1 = \{ r_1, r_2, r_4\}$ and $e_2 = \{ r_2, r_3\}$. Furthermore, since $\mbox{ker}(\Gamma) \cap \mathbb{R}_{> 0}^r \not= \emptyset$, we have that $\mathcal{E}$ covers $\R$.

For the CRN in Figure \ref{figure1} (left) $e_1$ does not correspond to a cycle but $e_2$ does so that $e_1$ is a stoichiometric generator of the CRN, while $e_2$ is a cyclic generator. For the CRN in Figure \ref{figure1} (right), we have that both $e_1$ and $e_2$ correspond to cycles so that $e_1$ and $e_2$ are both \emph{cyclic} generators of the CRN. The structural translation scheme \eqref{histidine-translation} therefore converted the stoichiometric generator $e_1$ into a cyclic generator. The primary objective of the methods presented in Section \ref{sec:main} will be to use structural translation to convert stoichiometric generators into a cyclic generators. Notably, if all of the stoichiometric generators are converted into cyclic generators then the deficiency of the resulting network is zero.
\end{exa}

\section{Main Results}
\label{sec:main}

In general, given a CRN $(\Sp,\C,\R)$, a structural translation $(\Sp,\C',\R')$ with desirable properties (e.g. weak reversibility, deficiency zero) is not known and therefore must be constructed. For biochemical reaction networks of realistic size, computational implementation is necessary.

Computational algorithms using mixed-integer linear programming (MILP) have been explored recently in \cite{J2,Tonello2017}. In \cite{J2}, the author presented a MILP program for performing network translation by reconstructing the reaction graph of the original network. The method, however, depended upon the translated network's complexes and the network's rate constants, both of which are typically not a priori known. The method introduced in \cite{Tonello2017}, by contrast, relies only upon knowledge of the network's elementary flux modes and attempts to convert the network's stoichiometric generators into cyclic generators. The method, however, requires a large number of decision variables and relies sensitively on the ordering of the reactions.

In this section, we present a novel computational method by which to compute structural translations. Our method depends upon a new CRN object which we call a \emph{reaction-to-reaction graph}. We show how this object relates to the underlying CRN and then introduce a \emph{binary linear programming} (BLP) problem on this graph which can be used to establish structural translations. This represents a significant improvement over existing methods since BLP problems can be solved in polynomial time in the number of constraints by Lenstra's algorithm \cite{Lenstra1983}.

\subsection{Reaction-to-Reaction Graph}

We introduce the following.

\begin{dfn}
A directed graph $G^\R = (V^\R, E^\R)$ is a \emph{reaction-to-reaction graph} of a CRN $(\Sp,\C,\R)$ if $V^\R = \R$ and $E^\R = \R \times \R$. 
Furthermore, we say that $(\Sp,\C,\R)$ and $G^\R$ are:
\begin{enumerate}
\item
\emph{product-to-source compatible} (PS-compatible) if, for any $r_i = y_{s(i)} \to y_{p(i)}$ and $r_j = y_{s(j)} \to y_{p(j)}$, $(r_i,r_j) \in E^\R$ if and only if $y_{p(i)} = y_{s(j)}$.
\item
\emph{common source compatible} (CS-compatible) if $y_{s(i)} = y_{s(j)}$ and $(r_k,r_i) \in E^\R$ implies $(r_k,r_j) \in E^\R$, i.e. if $r_i$ and $r_j$ have a common source complex then every reaction $r_k$ with an edge to $r_i$ has an edge to $r_j$.
\item
\emph{elementary flux mode compatible} (EM-compatible) if every minimal directed cycle in $G^\R$ corresponds to an elementary flux mode of $(\Sp,\C,\R)$.
\end{enumerate}
\end{dfn}


A reaction-to-reaction graph treats the reactions of a network as its vertices while the edges enforce additional conditions on the relationship with the underlying CRN (PS-, CS-, or EM-compatibility). The condition of PS-compatibility makes a correspondence between edges $(r_i, r_j) \in E^\R$ in the reaction-to-reaction graph and junctions of the following form in the reaction graph of the CRN:
\[\cdots \stackrel{r_i}{\longrightarrow} y \stackrel{r_j}{\longrightarrow} \cdots\]
The condition of CS-compatibility joins reactions from common source complexes, e.g.
\[\cdots \stackrel{r_k}{\longrightarrow} y \hspace{-0.05in} \begin{array}{c} {}^{r_i} \\[0.1in] {}_{r_j} \end{array} \hspace{-0.2in} \begin{array}{c}\nearrow \\[0.1in] \searrow\end{array} \]
forces $(r_k,r_i) \in E^\R$ and $(r_k,r_j) \in E^\R$. The condition of EM-compatibility forces a correspondence between elementary flux modes in the CRN and directed cycles in the reaction-to-reaction graph, although we note that the order of the reactions is not fixed in the reaction-to-reaction graph.


Our goal is to construct reaction-to-reaction graphs which are CS- and EM-compatible with a given CRN $(\Sp,\C,\R)$, and then enforce PS-compatibility to construct a network translation $(\Sp,\C',\R')$. Consider the following examples.




\begin{figure}[h!]
    \centering
    \begin{subfigure}[t]{0.35\textwidth}
        \centering
        \includegraphics[height=1.5in]{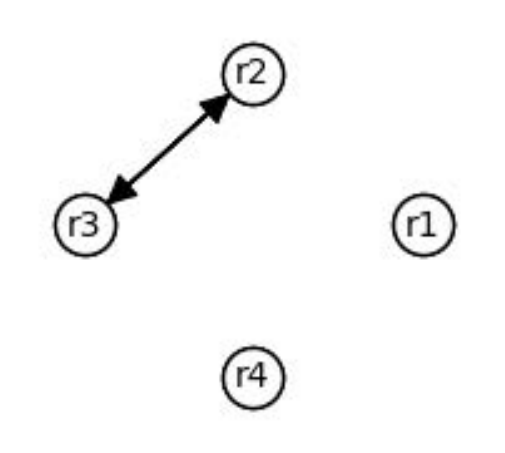}
    \end{subfigure}%
    ~ 
    \begin{subfigure}[t]{0.35\textwidth}
        \centering
        \includegraphics[height=1.5in]{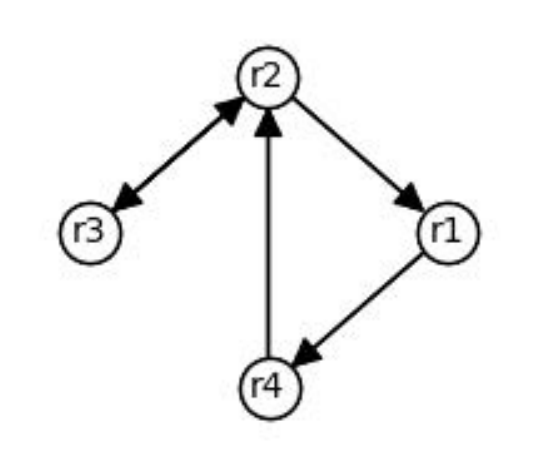}
    \end{subfigure}
    \caption{Two reaction-to-reaction graphs of four reactions, labeled $\{ r_1, r_2, r_3, r_4\}$. The reaction-to-reaction graph (a) is PS- and CS-compatible with the CRN in Figure \ref{figure1} (left) while the reaction-to-reaction graph (b) is EM- and CS-compatible with the CRN in Figure \ref{figure1} (left) and PS-, CS-, and EM-compatible with the CRN in Figure \ref{figure1} (right).}
    \label{figure2}
\end{figure}

\begin{exa}
\label{example5}
Reconsider the histidine network given in Figure \ref{figure1} (left). We can construct a reaction-to-reaction graph which is PS- and CS-compatible, but not EM-compatible, with this CRN by selecting the edges $E^\R = \{ (r_2,r_3), (r_3,r_2) \}$. This gives the reaction-to-reaction graph in Figure \ref{figure2} (left). Note that we do not include any interactions involving $r_1$ and $r_4$ since these reactions do not connect with any others in the reaction graph of the CRN.

Alternatively, we may construct a reaction-to-reaction graph which is EM- and CS-compatible but not PS-compatible to the CRN in Figure \ref{figure1} (left). We select edges such that there are minimal cycles on $e_1 = \{ r_1, r_2, r_4\}$ and $e_2 = \{ r_2, r_3\}$. Selecting $E^\R = \{ (r_1,r_2),(r_2,r_4),(r_4,r_1),(r_2,r_3),(r_3,r_2)\}$ gives the reaction-to-reaction graph given in Figure \ref{figure1} (right). It can be checked exhaustively that there is no reaction-to-reaction graph which is all of PS-, CS-, and EM-compatible with this CRN.
\end{exa}

\begin{exa}
\label{example6}
Consider the CRN in Figure \ref{figure1} (right). We may construct a reaction-to-reaction graph which is PS-, CS-, and EM-compatible with the CRN in Figure \ref{figure1} (right) by taking $E^\R = \{ (r_1,r_2),(r_2,r_4),(r_4,r_1),(r_2,r_3),(r_3,r_2)\}$. Notably, this edge set coincides the edge set for the reaction-to-reaction graph which was CS- and EM-compatible to the CRN in Figure \ref{figure1} (left).

\end{exa}

\begin{rem}
Notice that the implications in CS-compatibility are one directional only, i.e. it is not necessary that $(r_k,r_i) \in E^\R$ and $(r_k,r_j) \in E^\R$ imply $y_{s(i)} = y_{s(j)}$. For example, consider the reaction-to-reaction graph in Figure \ref{figure2} (right) as it relates to the CRN in Figure \ref{figure1} (left). We have that $(r_1,r_2) \in E^\R$ and $(r_3,r_2) \in E^\R$; however, we have $y_{s(1)} = X \not= X + Y_p = y_{s(3)}$.
\end{rem}

\subsection{Main Theory}
\label{sec:theory}

In order to state our objectives in Section \ref{sec:blp}, we need to further understand the relationship between CRNs and PS-, CS-, and/or EM-compatibility of reaction-to-reaction graphs.

We have the following results.

\begin{lem}
\label{lem1}
Consider a CRN $(\Sp,\C,\R)$ and a reaction-to-reaction graph $G^\R = (V^\R, E^\R)$ which is PS-compatible with $(\Sp,\C,\R)$. Then $G^\R$ is CS-compatible with $(\Sp,\C,\R)$.
\end{lem}

\begin{proof}
Consider a CRN $(\Sp,\C,\R)$ and let $G^\R = (V^\R, E^\R)$ be a reaction-to-reaction graph which is PS-compatible with $(\Sp,\C,\R)$. Suppose that $(r_k,r_i) \in E^\R$ and $y_{s(i)} = y_{s(j)}$ for some $r_j \in \R$. Since $G^\R$ is PS-compatible with $(\Sp,\C,\R)$, we have $y_{p(k)} = y_{s(i)} = y_{s(j)}$. It follows from PS-compatibility that $(r_k,r_j) \in E^\R$ and therefore CS-compatibility is attained.
\end{proof}

\begin{lem}
\label{lemma1}
Consider a CRN $(\Sp,\C,\R)$ and a reaction-to-reaction graph $G^\R = (V^\R, E^\R)$  which is PS-compatible with $(\Sp,\C,\R)$. Suppose $(\Sp,\C,\R)$ has a set of elementary modes $\mathcal{E} = \{e_1, \ldots, e_p \}$ which is unitary and covers $\R$. Then $G^\R$ is EM-compatible with $(\Sp,\C,\R)$ if and only if $(\Sp,\C,\R)$ is weakly reversible and deficiency zero.
\end{lem}

\begin{proof}
Consider a CRN $(\Sp,\C,\R)$ and let $G^\R = (V^\R, E^\R)$ be a reaction-to-reaction graph which is PS-compatible with $(\Sp,\C,\R)$. Note that PS-compatibility implies CS-compatibility by Lemma \ref{lem1}. We prove the forward and backward implications separately.\\

\noindent ($\Longrightarrow$) Suppose $G^\R$ is EM-compatible with $(\Sp,\C,\R)$. Since the elementary modes of $(\Sp,\C,\R)$ cover $\R$, we have by EM-compatibility that every reaction (vertex) in $G^\R$ is a part of a cycle. It follows immediately that $(\Sp,\C,\R)$ is weakly reversible.

Now suppose that $(\Sp,\C,\R)$ is not deficiency zero. It follows that $\delta = \mbox{dim}(\mbox{ker}(Y) \cap \mbox{im} (I_a)) > 0$ so that there is a vector $\mathbf{v} \in \mathbb{R}^r$ such that $I_a \mathbf{v} \not= \mathbf{0}$ but $Y I_a \mathbf{v} = \mathbf{0}$. If $\mathbf{v} \in \mathbb{R}_{\geq 0}^r$, since $\Gamma \mathbf{v} = \mathbf{0}$, we have that $\mathbf{v} \in \mathcal{E}$, i.e. $\mathbf{v}$ is in the elementary flux cone. Since the elementary modes are unitary, we have that $\mathbf{v}$ corresponds to a summation of cycles in $(\Sp,\C,\R)$ so that $I_a \mathbf{v} = \mathbf{0}$, which is a contradiction.

Now suppose that $\mathbf{v} \not\in \mathbb{R}_{\geq 0}^r$, i.e. at least two components have opposite signs. Then, since $\mathcal{E}$ covers $\R$, we have that there are $\lambda_i \geq 0$, $i = 1, \ldots, p,$ such that $\mathbf{w} = \mathbf{v} + \sum_{i=1}^p \lambda_i e_i \in \mathbb{R}_{\geq 0}^r$. Furthermore, we have $\Gamma \mathbf{w} = \Gamma \mathbf{v} + \sum_{i=1}^p \lambda_i \Gamma e_i = \mathbf{0}$ so that $\mathbf{w} \in \mathcal{E}$. Since the elementary flux modes are unitary, it follows that $\mathbf{w}$ corresponds to a summation of cycles in $(\Sp,\C,\R)$ so that $I_a \mathbf{w} = \mathbf{0}$. Note that $I_a \mathbf{v} = I_a \mathbf{w} - \sum_{i=1}^r \lambda_i I_a e_i = \mathbf{0}$. This a contradicts our assumptions and completes the forward direction of the proof.\\

\noindent ($\Longleftarrow$) Now suppose that $(\Sp,\C,\R)$ is weakly reversible and deficiency zero. It follows from $\delta = 0$ that $\Gamma \mathbf{v} = \mathbf{0}$ implies $I_a \mathbf{v} = \mathbf{0}$, i.e. if $\mathbf{v} \in \mathcal{E}$ then $\mathbf{v}$ is a cyclic generator of the CRN. It follows from PS-compatibility that every elementary flux mode is a cycle in the reaction graph of the CRN, and therefore a cycle in $G^\R$. It follows that $G^\R$ is EM-compatible, and we are done.
\end{proof}

We now want to relate the properties of PS-, CS-, and EM-compatibility to structural translation (Definition \ref{def:translation}). We have the following result.

\begin{thm}
\label{thm:main}
Consider a CRN $(\Sp,\C,\R)$ with a set of elementary flux modes $\mathcal{E} = \{e_1, \ldots, e_p\}$ which is unitary and covers $\R$. If there is a reaction-to-reaction graph $G^\R = (V^\R, E^\R)$  which is CS- and EM-compatible to $(\Sp,\C,\R)$ then there is a CRN $(\Sp,\C',\R')$ which is PS-, CS-, and EM-compatible with $G^\R$. Furthermore, $(\Sp,\C',\R')$ is a weakly reversible, zero deficiency structural translation of $(\Sp,\C,\R)$. In particular, the translation complexes $\alpha_i \in \mathbb{R}_{\geq 0}^m$, $i=1,\ldots, r$, required to produce such a translation satisfy the following linear system, which is necessarily consistent:
\begin{equation}
\label{eq:3}
\alpha_i - \alpha_j = y_{s(j)} - y_{p(i)}, \hspace{0.25in} (r_i,r_j) \in E^\R.
\end{equation}
\end{thm}

\begin{proof}
Consider a reaction-to-reaction graph $G^\R = (V^\R, E^\R)$ which is CS- and EM-compatible with $(\Sp,\C,\R)$. We show that it is possible to construct a CRN $(\Sp,\C',\R')$ which is EM-, CS-, and PS-compatible to $G^\R$ by setting up and solving the corresponding linear system \eqref{eq:3} in the translation complexes $\Lambda = \{ \alpha_1, \ldots, \alpha_r \}$.

In order for $(\Sp, \C', \R')$ to be PS-compatible to $G^\R$, we require that
\begin{equation}
\label{eq:1}
y'_{p(i)} = y'_{s(j)}, \hspace{0.25in} \mbox{for } (r_i,r_j) \in E^\R.
\end{equation}
Note that we can satisfy this set of equations if there is a set of translation complexes $\Lambda = \{ \alpha_1, \ldots, \alpha_r \}$ where $\alpha_i \in \mathbb{R}^m$, $i =1, \ldots, r$, such that $y'_{p(i)} = y_{p(i)} + \alpha_i$ and $y'_{s(j)} = y_{s(j)} + \alpha_j$, i.e. each complex in $\C'$ results from translating a complex in $\C$ by the corresponding translation complex $\alpha_i \in \Lambda$. From \eqref{eq:1}, this gives the system
\begin{equation}
\label{eq:2}
y_{p(i)} + \alpha_i  = y_{s(j)} + \alpha_j, \hspace{0.25in} (r_i,r_j) \in E^\R
\end{equation}
which can be rearranged to give \eqref{eq:3} in the unknown vectors $\alpha_i \in \mathbb{R}^{m}$, $i = 1, \ldots, r$. We now show that, since $G^\R$ is CS- and EM-compatible with $(\Sp,\C,\R)$, that \eqref{eq:3} is necessarily consistent.

For ease of notation, we let $q = |E^\R|$ and suppose the edges $(r_i,r_j) \in E^\R$ are ordered $1, \ldots, q$, i.e. $E^\R = \{ v_1, \ldots, v_q \} \subseteq \R \times \R$ where $v_k = (r_i, r_j)$. We can then write \eqref{eq:3} as the linear system $A \alpha = \mathbf{b}$ where $\alpha = ( \alpha_1, \ldots, \alpha_r) \in \mathbb{R}^{mr}$ is a vector of unknowns, $\mathbf{b} = (b_1, \ldots, b_q) \in \mathbb{R}^{mq}$ has entries $b_k = (y_{s(j)} - y_{p(i)}) \in \mathbb{R}^m$ if $v_k = (r_i,r_j) \in E^\R$, and $A \in \mathbb{R}^{mq \times mr}$ has the block structure
\begin{equation}
\label{eq:4}
A = \left[ \begin{array}{cccc} A_{11} & A_{12} & \cdots & A_{1r} \\ A_{21} & A_{22} & \cdots & A_{2r} \\ \vdots & \vdots & \ddots & \vdots \\ A_{q1} & A_{q2} & \cdots & A_{qr} \end{array} \right]
\end{equation}
where, given $v_k = (r_i, r_j) \in E^\R$, we set $A_{ki} = -I^{m \times m}$, $A_{kj} = I^{m \times m}$, and $A_{kl} = \mathbf{0}^{m \times m}$ for all $l \not= i$ or $l \not= j$.

In order to show the linear system $A \alpha = \mathbf{b}$ is consistent, it is sufficient to show that $\mathbf{c} \in \mbox{ker}(A^T)$ implies that $\mathbf{c}^T \mathbf{b} = 0$. To characterize $\mbox{ker}(A^T)$, notice that the block structure of $A^T$ corresponds to the incidence matrix of $G^\R$ (interpreting the identity blocks $I^{m \times m}$ as $1$ and the $0^{m \times m}$ blocks as $0$). It follows that $\mbox{ker}(A^T)$ has support on the minimal cycles of $G^\R$ which correspond by EM-compatibility to the elementary modes $\mathcal{E} = \{ e_1, \ldots, e_p \}$ of $(\Sp,\C,\R)$. 
We can extend this to the block structure of $A^T$ in the following way: to each elementary mode $e_k \in \mathcal{E}$, we introduce an arbitrary vector $\bar{\beta}_k \in \mathbb{R}^r$ and define $\beta = (\beta_1, \ldots, \beta_q) \in \mathbb{R}^{mq}$ such that, if $\{ v_{\mu(1)},\ldots, v_{\mu(l)}\}$ the minimal cycle in $G^\R$ corresponding to the elementary mode $e_k$, we have $\beta_{\mu(i)} = \bar{\beta}_k$ for all $i = 1, \ldots, l$. We have that $\{ \beta_1, \ldots, \beta_k \}\subseteq \mathbb{R}_{mq}$ forms a basis of $\mbox{ker}(A^T)$. 
Furthermore, it follows that
\[\beta_k^T \mathbf{b} = \bar{\beta}_k^T \sum_{i=1}^l (y_{s(\mu(i))} - y_{p(\mu(j))}) = -\bar{\beta}_k^T \sum_{i=1}^l (y_{p(\mu(i))} - y_{s(\mu(i))}) = 0\]
since $e_k = \{r_{\mu(1)} , \ldots, r_{\mu(l)}\}$ is an elementary flux mode of $(\Sp,\C,\R)$.

It follows that the system $A \alpha = \mathbf{b}$ is consistent so that we may solve the system \eqref{eq:3} for the translation complexes $\Lambda = \{ \alpha_1, \ldots, \alpha_r\}$. By construction, the resulting network $(\Sp,\C',\R')$ is PS-, CS-, and EM-compatible with $G^\R$. Furthermore we have $\Gamma = \Gamma'$ so that $(\Sp,\C,\R)$ and $(\Sp,\C',\R')$ are structural translations of one another, and $(\Sp,\C',\R')$ is weakly reversible and deficiency zero by Lemma \ref{lemma1}, and we are done.
\end{proof}

\begin{exa}
Reconsider the reaction-to-reaction graph in Figure \ref{figure2} (right). The reaction-to-reaction graph in Figure \ref{figure1} (right) is both CS- and EM-compatible with the CRN in Figure \ref{figure1} (left). It follows from Theorem \ref{thm:main} that there is a CRN $(\Sp,\C',\R')$ which is PS-, CS-, and EM-compatible with the reaction-to-reaction graph in Figure \ref{figure2}(right). Furthermore, this CRN is a weakly reversible, deficiency zero structural translation of the original CRN. We can quickly verify that these properties are satisfied by the CRN in Figure \ref{figure1} (right).
\end{exa}

\begin{rem}
Note that having a structural translation with a deficiency of zero corresponds to a stoichiometric deficiency of zero in the corresponding generalized chemical reaction network \cite{MR2012,MR2014}. It is still possible, however, that the kinetic-order deficiency is nonzero (see Appendix \ref{app:mapk}).
\end{rem}

\subsection{Computing Structural Translations}
\label{sec:blp}

Theorem \ref{thm:main} and the networks in Figure \ref{figure1} suggests a process by which to construct structural translations. We perform the following steps:
\begin{enumerate}
\item
From the given CRN $(\Sp,\C,\R$), we compute the set of elementary flux modes $\mathcal{E} = \{e_1, \ldots, e_p \}$ and the set of reactions with shared source complexes, i.e. $\mathcal{F} = \{ f_1, \ldots, f_q\}$ where $r_i, r_j \in f_k$, $i \not= j$, for some $k$ if $y_{s(i)} = y_{s(j)}$.
\item
From the sets $\mathcal{E}$ and $\mathcal{F}$, we determine a reaction-to-reaction graph $G^\R = (V^\R, E^\R)$ which is CS- and EM-compatible with $(\Sp,\C,\R)$.
\item
From this reaction-to-reaction graph $G^\R$, we construct a CRN $(\Sp,\C',\R')$ which is PS-, CS-, and EM-compatible with $G^\R$ by solving \eqref{eq:3}.
\end{enumerate}

\noindent Note that, if successful, this algorithm produces a weakly reversible, deficiency zero structural translation of $(\Sp,\C,\R)$ by Theorem \ref{thm:main}. In what follows we describe the approaches taken to these three steps.

\subsubsection*{Step 1: Computing Elementary Flux Modes}

Consider a CRN $(\Sp,\C,\R)$. To determine the set of elementary flux modes $\mathcal{E}$ of $(\Sp,\C,\R)$, we use the {\tt crnpy} Python package developed by Elisa Tonello \cite{Tonello2016crnpy}. In accordance with Theorem \ref{thm:main}, we do not consider the network if the set $\mathcal{E}$ is not unitary (i.e. some elementary flux modes $e_i \in \mathcal{E}$ with entries other than zeros and ones) or does not cover $\R$ (i.e. there is a reaction which does not have support on any elementary mode $e_i \in \mathcal{E}$).

We also collect sets of reactions with shared source complexes into a set $\mathcal{F} = \{ f_1, \ldots, f_q \}$ where $q$ is the number of source complexes which are the source for at least two reactions. This set can be constructed by direct analysis of the incidence matrix $I_a$ of the CRN.

Throughout this section, we consider elementary flux modes according to their supports, i.e. $e_i \subseteq \R$.


\subsubsection*{Step 2: Computing the Reaction-to-Reaction Graph}

Recall that a \emph{binary linear programming} (BLP) problem can be stated in the general form
\begin{equation}
\label{blp}
\begin{array}{rl}
\mbox{maximize} &  \mathbf{c}^T \mathbf{x}\\
\mbox{subject to} & A \mathbf{x} \leq \mathbf{b}
\end{array}
\end{equation}
where $A \in \mathbb{R}^{n \times m}$, $
\mathbf{b} \in \mathbb{R}^n$, and $\mathbf{c} \in \mathbb{R}^m$ are matrices and vectors of parameters, and $\mathbf{x} \in \{0, 1\}^m$ is a vector of binary decision variables. 

We formulate the problem of determining a reaction-to-reaction graph $G^\R = (V^\R, E^\R)$ which is CS- and EM-compatible with $(\Sp,\C,\R)$ as a BLP problem. We introduce binary decision variables $x_{ij} \in \{ 0, 1 \}$, $i, j = 1, \ldots, r$, $i \not= j$, with the following logical requirement:
\[x_{ij} = 1 \; \Longleftrightarrow \; (r_i, r_j) \in E^\R\]
where $E$ is the edge set of our reaction-to-reaction graph $G^\R = (V^\R, E^\R)$. We now seek to set up constraints sufficient to guarantee the reaction-to-reaction graph $G^\R$ is CS- and EM-compatible with $(\Sp,\C,\R)$. For this purpose, it is sufficient to consider the sets $\mathcal{E}$ and $\mathcal{F}$ determined in Step 1.\\

\noindent \emph{Elimination of unnecessary edges:} It is often apparent from the structure of $\mathcal{E}$ and $\mathcal{F}$ that the reactions may be partitioned into noninteracting sets of reactions. We use the following rules to establish these partitions:
\begin{enumerate}
\item
$e_i \equiv e_j$ if $e_i \cap e_j \not= \emptyset$, and
\item
$e_i \equiv e_j$ if there are $r_k \in e_i$ and $r_l \in e_j$ such that $r_k, r_l \in f_p$ for some $f_p \in \mathcal{F}$.
\end{enumerate}
That is, two elementary modes are connected if they share a reaction (condition 1) or possess reactions which have a common source complex (condition 2). To define the desired partitions of the reactions, we take the transitive closure of the $\equiv$ operation defined above, and then the union of the reactions in each equivalence class of elementary flux modes. This gives a set
\[\mathcal{G} = \{ g_1, \ldots, g_w \}, \; \mbox{ where } \; g_k = \bigcup_{e_i \equiv e_j} e_i.\]
We then impose the following partition rule:
\begin{equation}
\label{const:connected}\tag{\textbf{Par}}
\begin{aligned}
x_{ij}& = 0, \hspace{0.5in} & \mbox{if } r_i \in g_k \mbox{ and } r_j \not\in g_k \mbox{ for any } k =1, \ldots, w.
\end{aligned}
\end{equation}\hspace{0.1in}

\noindent \emph{CS-compatibility:} To guarantee $G^\R = (V^\R, E^\R)$ is CS-compatible with $(\Sp,\C,\R)$, we impose that, if $r_i, r_j \in f_l$ for some $f_l \in \mathcal{F}$, then
\begin{equation}
\label{const:source}\tag{\textbf{CS}}
 \begin{aligned}
x_{ki} - x_{kj} & = 0, \hspace{0.5in} & k = 1, \ldots, r.
\end{aligned}
\end{equation}
The constraint set \eqref{const:source} guarantees that either $(r_k, r_i) \in E^\R$ and $(r_k, r_j) \in E^\R$, or $(r_k, r_i) \not\in E^\R$ and $(r_k, r_j) \not\in E^\R$.\\

\noindent \emph{EM-compatibility:} Consider an elementary flux mode $e_k \in \mathcal{E}$ and define $l(k) = |e_k|$. We introduce the following constraint set:
\begin{equation}
\label{const:cycle}\tag{\textbf{EM1}}
\left\{ \begin{aligned}
\sum_{r_i, r_j \in e_k} x_{ij} & = l(k) & \\
\sum_{r_i \in e_k} x_{ij} & = 1, \hspace{0.5in} & r_j \in e_k \\
\sum_{r_j \in e_k} x_{ij} & = 1, \hspace{0.5in} & r_i \in e_k.
\end{aligned} \right.
\end{equation}
The first constraint set in \eqref{const:cycle} guarantees that the number of edges on a component corresponding to the support of an elementary flux mode contains exactly the number of edges contained in the elementary flux mode. The second constraint set in \eqref{const:cycle} guarantees that each vertex of the component has exactly one outgoing edge, while the third constraint set guarantees that each such vertex has exactly one incoming edge.

The constraint set \eqref{const:cycle} guarantees that every vertex (reaction) with support on a given elementary mode is a part of exactly one cycle on the support of that elementary mode. It does not, however, guarantee that these cycles are maximal with respect to the support of the elementary mode. For example, an elementary mode consisting of $6$ reactions may be split into a $2$-cycle and a $4$-cycles, or two $3$-cycles. We furthermore impose that elementary flux modes may not be decomposed into subcycles. We guarantee this by imposing that, for every elementary mode $e_k \in \mathcal{E}$ with $l(k) = |e_k| \geq 4$, every combination $e'_k = \{ r_{\mu(1)}, \ldots, r_{\mu(l(k'))} \} \subset e_k$ with $2 \leq l(k') = |e'_k| \leq \lfloor \frac{k}{2} \rfloor$ satisfies:
\begin{equation}
\label{const:subcycle}\tag{\textbf{EM2}}
\left\{ \begin{aligned}
\sum_{r_i, r_j \in e'_k} x_{ij} \leq l(k') - 1
\end{aligned} \right.
\end{equation}
Since a cycle on a component of size $l(k')$ is required to have $l(k')$ edges, the constraint set \eqref{const:subcycle} guarantees that no subcycles exist on the support of an elementary flux mode. Notice that we do not need to apply this condition for components $l(k') > \lfloor \frac{l(k)}{2} \rfloor$ since a subcycle of such size necessitates a subcycle of size $l(k') \leq \lfloor \frac{l(k)}{2} \rfloor$ by \eqref{const:cycle}.\\

\noindent \emph{Objective function:} 
We impose the following objective function
\begin{equation}
\label{objectivefunction}\tag{\textbf{Obj}}
\mbox{minimize} \; \; \mathop{\sum_{i,j=1}^r}_{i \not= j} x_{ij}.
\end{equation}
That is, we minimize the number of edges in $(r_i,r_j) \in E^\R$. This prohibits the procedure from adding unnecessary edges $(r_i,r_j) \in E^\R$. We produce a reaction-to-reaction graph by optimizing \eqref{objectivefunction} over the constraint sets \eqref{const:connected}, \eqref{const:cycle}, \eqref{const:subcycle}, and \eqref{const:source}. 

\begin{rem}
Although \eqref{const:subcycle} guarantees that there are no subcycles on a given elementary flux mode, it is possible that the optimization procedure will create cycles which do not correspond to minimal elementary flux modes. The resulting reaction-to-reaction graph will then fail to be EM-compatible with $(\Sp,\C,\R)$. Rather than implement further constraints like \eqref{const:subcycle} to eliminate this possibility, we note that such a network will fail to have a consistent system \eqref{eq:3}. Consistency of a linear system $A \mathbf{x} = \mathbf{b}$ is simple to check computationally by checking $\mbox{rank}(A) = \mbox{rank}(B)$ where $B$ is the augmented matrix $B = [A \mid \mathbf{b}]$. If $\mbox{rank}(A) \not= \mbox{rank}(B)$, we do not proceed to Step 3.
\end{rem} 



\subsubsection*{Step 3: Construct structural translation $(\Sp,\C',\R')$}

To construct a structural translation $(\Sp,\C',\R')$ from the reaction-to-reaction graph produced in Step 2, we need to solve the linear system \eqref{eq:3}. As a preprocessing step, we check whether \eqref{eq:3} is consistent by computing the rank of the associated matrices. If the system is not consistent, the network does not admit a structural translation by Lemma \ref{lemma1}. If the system is consistent, we may construct a structural translation $(\Sp,\C',\R')$ by solving \eqref{eq:3} for the set of translation complexes $\Lambda = \{ \alpha_1, \ldots, \alpha_r\}$.

Rather than solve \eqref{eq:3} directly we use the observation that, for a known $\alpha_i$, we have
\[\alpha_j  = y_{p(i)} - y_{s(j)} + \alpha_i\]
for every $r_j \in \R$ such that $(r_i, r_j) \in E^\R$. Consequently, we may use the following algorithm to solve \eqref{eq:3}:
\begin{enumerate}
\item
Initialize the sets $\mathcal{P} = \R$, $\mathcal{P}' = \emptyset$, and $\mathcal{P}'' = \emptyset$.
\item
Select an arbitrary $r_i \in \mathcal{P}$ and then:
\begin{enumerate}
\item
set $\alpha_i = \mathbf{0}$ and
\item
set $\mathcal{P}' = \{ r_i\}$ and $\mathcal{P} = \mathcal{P} \setminus  \{r_i \}$.
\end{enumerate}
\item
For all $r_j \in \mathcal{P}$ such that $r_i \in \mathcal{P}'$ and $(r_i, r_j) \in E^\R$, do the following:
\begin{enumerate}
\item
set $\alpha_j = y_{p(i)}-y_{s(j)}+\alpha_i$ and
\item
set $\mathcal{P}'' = (\mathcal{P}'' \cup \{r_j\}) \setminus \mathcal{P}'$.
\end{enumerate}
\item
If $\mathcal{P}'' \not= \emptyset$, then:
\begin{enumerate}
\item
set $\mathcal{P}' = \mathcal{P}''$, $\mathcal{P} = \mathcal{P} \setminus \mathcal{P}'$, and $\mathcal{P}'' = \emptyset$ and
\item
repeat from step 3.
\end{enumerate}
\item
If $\mathcal{P}'' = \emptyset$ and $\mathcal{P} \not= \emptyset$ then repeat from step 2.
\item
If $\mathcal{P}'' = \emptyset$ and $\mathcal{P} = \emptyset$, we are done.
\end{enumerate}
This algorithm solves for each translation complex in \eqref{eq:3} successively and can in general be solved more efficiently than the corresponding system in matrix form. We subsequently adjust translation complexes so that the resulting complexes are nonnegative by adding nonnegative complexes to entire linkage classes where needed.

\section{Examples}
\label{sec:examples}

In this section, we apply the algorithm presented in Section \ref{sec:blp} to 508 curated models from the European Bioinformatics' BioModels Database and summarize the output. We also expand upon two models the algorithm determined to have a weakly reversible, zero deficiency structural translation: a zigzag model of plant-pathogen interactions \cite{Pritchard2014,Jones2006}, and a MAPK cascade model \cite{M-H-K}. In Appendix \ref{app:a}, we outline how the outcome of the algorithm in Section \ref{sec:blp} can be utilized to construct steady state parametrizations according to \cite{J-M-P} and establish mono or multistationarity within stoichiometric compatibility classes according to \cite{C-F-M-W2016}.

\subsection{BioModels Database}

We implemented the algorithm outlined in Section \ref{sec:blp} in Python and tested it on 508 curated networks from the European Bioinformatics Institute's Biomodels database \cite{BioModels}. We imposed a twenty minute timeout per model. The algorithm found 176 models which permitted a weakly reversible, deficiency zero structural translation to be computed. Of those models, 34 were not originally weakly reversible, deficiency zero networks.

Of the models for which the program did not succeed in finding a weakly reversible, deficiency zero structural translation, 239 failed because the network had an elementary flux mode set $\mathcal{E}$ which either was not unitary or did not cover $\R$, 60 failed because a EM- and CS-compatible reaction-to-reaction graph could not be constructed, and 27 failed due to computational time out. The mean size of the networks which failed to compute due to computational timeout was 387 reactions, and the median was 144 reactions.

\begin{figure}[t]
\begin{center}
\begin{subfigure}[b]{0.48\textwidth}
\includegraphics[width=\textwidth]{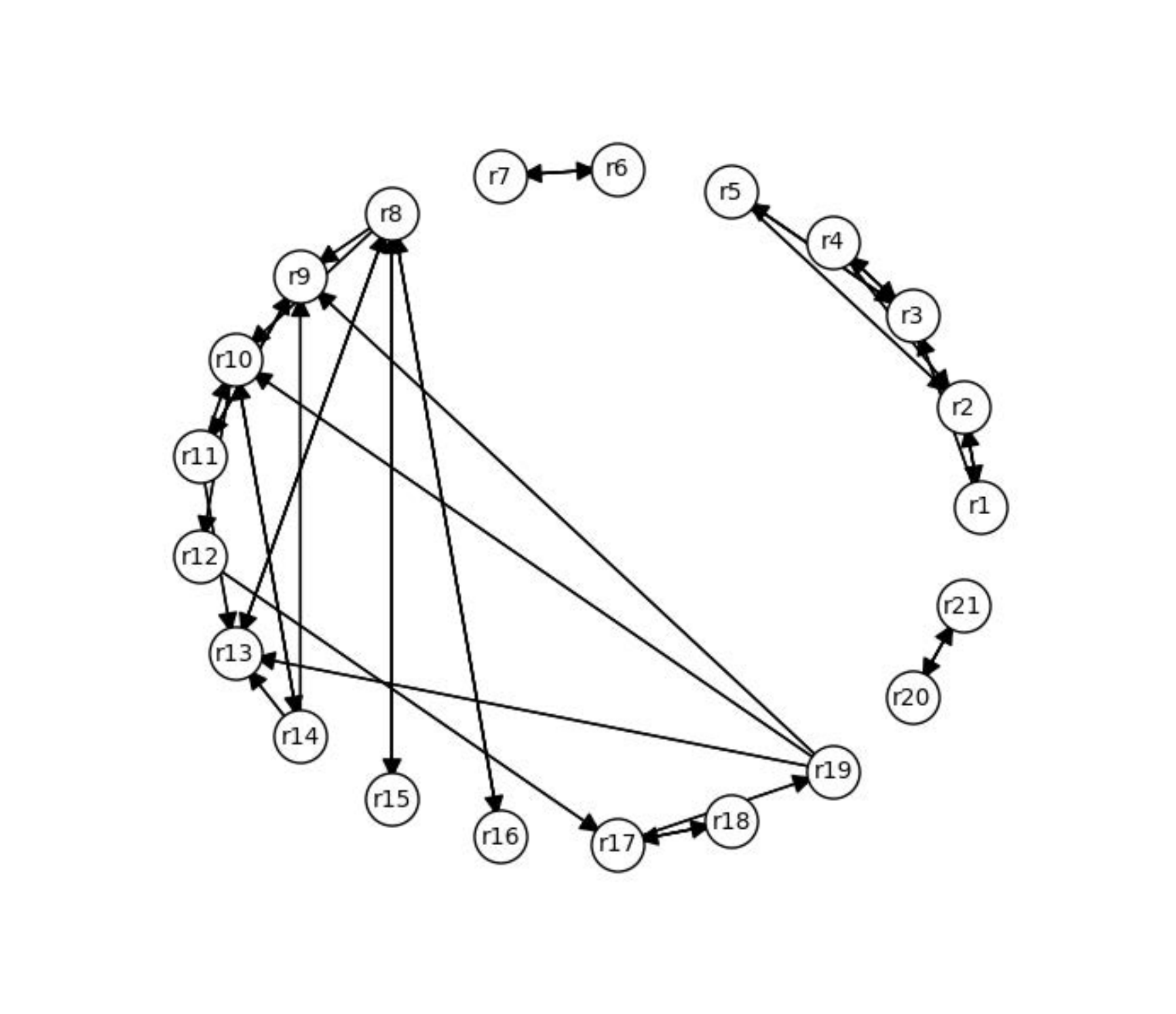}
        \caption{Reaction-to-reaction graph corresponding to \eqref{zigzag} and \eqref{zigzagtranslated}.}
    \end{subfigure}
    ~ 
    \begin{subfigure}[b]{0.42\textwidth}
        \includegraphics[width=\textwidth]{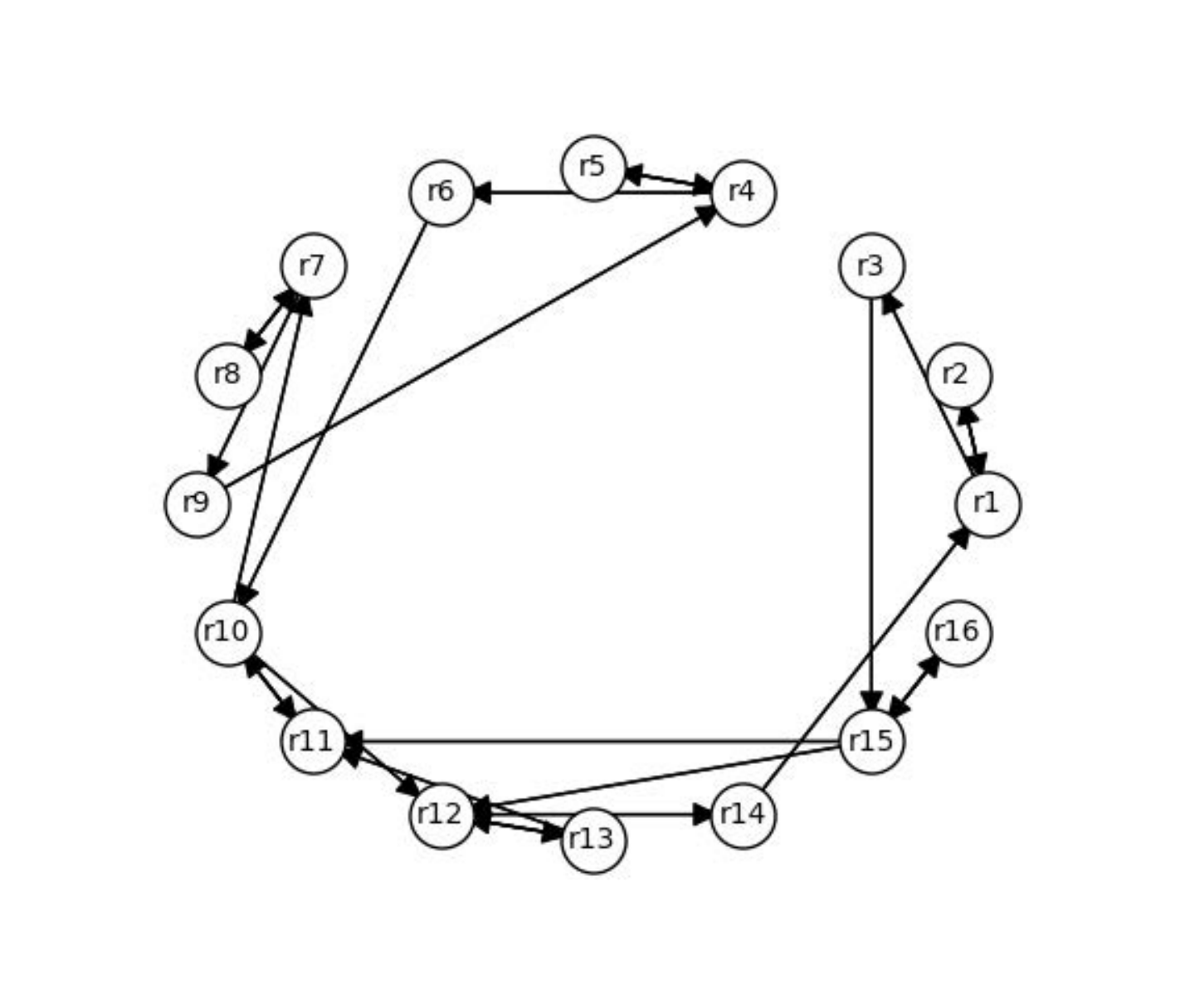}
        \caption{Reaction-to-reaction graph corresponding to \eqref{mapk} and \eqref{mapk-translated}.}
    \end{subfigure}
    \end{center}
\caption{Two reactions-to-reaction graphs determined by the BLP outlined in Section \ref{sec:blp}. In (a), the reaction-to-reaction graph is CS- and EM-compatible with \eqref{zigzag} and PS-, CS-, and EM-compatible with \eqref{zigzagtranslated}. In (b), the reaction-to-reaction graph is CS- and EM-compatible with \eqref{mapk} and PS-, CS-, and EM-compatible with \eqref{mapk-translated}.}
\label{zigzagr2r}
\end{figure}

\subsection{Example: Zigzag Model}
\label{sec:zigzag}

Consider the following network of the zigzag model of plant-pathogen interactions \cite{Pritchard2014,Jones2006} which corresponds to network {\tt biomd0000000563} in the BioModels database \cite{BioModels}:
\small
\begin{equation}
\label{zigzag}
\begin{tikzcd}
X_1 + X_2 \arrow[r,  yshift=+0.5ex, "r_1"] & X_3 \arrow[l,yshift=-0.5ex,"r_2"] & X_9 \arrow[r, "r_9"] & X_1 + X_9 & X_7 + X_9 \arrow[r, "r_{15}"] & X_7 \\[-0.15in]
X_3 \arrow[r, "r_3"] & X_3 + X_4 & X_9 \arrow[r, "r_{10}"] & X_{10} + X_9 & X_4 + X_9 \arrow[r, "r_{16}"] & X_4 \\[-0.15in]
X_4 \arrow[r, "r_4"] & \emptyset & X_{1} \arrow[r, "r_{11}"] & \emptyset & X_{10} + X_{11} \arrow[r, yshift=0.5ex, "r_{17}"] & X_{12} \arrow[l, yshift=-0.5ex, "r_{18}"] \\[-0.15in]
X_4 + X_5 \arrow[r, "r_5"] & X_5 & X_{5} \arrow[r, "r_{12}"] & \emptyset & X_{12} \arrow[r, "r_{19}"] & X_5 + X_{11} \\[-0.15in]
X_5 + X_6 \arrow[r, yshift=0.5ex,"r_6"] & X_7 \arrow[l, yshift=-0.5ex,"r_7"] & X_{9} \arrow[r, "r_{13}"] & \emptyset & X_{4} + X_{11} \arrow[r, yshift=0.5ex,"r_{20}"] & X_{13} \arrow[l,yshift=-0.5ex,"r_{21}"] \\[-0.15in]
X_8 \arrow[r, "r_8"] & X_8 + X_9 & X_{10} \arrow[r, "r_{14}"] & \emptyset & &
\end{tikzcd}
\end{equation}
\normalsize
where $X_1 = \mbox{PAMP}$, $X_2 = \mbox{PRR}$, $X_3 = \mbox{PRR}^*$, $X_4 = \mbox{Callose}$, $X_5 = \mbox{E}_{\mbox{int}}$, $X_6 = \mbox{R}$, $X_7 = \mbox{R}^*$, $X_8 = \mbox{Pathogen}_{\mbox{bulk}}$, $X_9 = \mbox{Pathogen}$, $X_{10} = \mbox{E}$, $X_{11} = \mbox{F}$, $X_{12} = \mbox{EF}$, and $X_{13} = \mbox{FCallose}$.

All interactions in \cite{Pritchard2014} are assumed to be mass-action except for $X_{10} \to X_{5}$ which is inhibited by $X_4$ according to the competitive inhibition reaction rate
\begin{equation}
\label{inhibition}
\frac{V_{max}x_{10}}{ \frac{K_m}{K_i}x_4 + x_{10} + K_m}
\end{equation}
where $V_{max}, K_i, K_m > 0$ are parameters.  We have replaced the reaction $X_{10} \to X_5$ with the reaction set $r_{17}$ through $r_{21}$ in \eqref{zigzag} 
to reflect the activity of an unseen activator ($X_{11} = \mbox{F}$) and inhibition of $X_{11}$ by $X_4$. The quasi-steady-state approximation for the production of $X_5$ is given by \eqref{inhibition} with $V_{max} = k_{20} (x_{11}(0) + x_{12}(0))$, $K_i = \frac{k_{23}}{k_{22}}$ and $K_m = \frac{k_{19}+k_{20}}{k_{18}}$ \cite{Ingalls}. Consequently, the steady states of the mass-action system we use and the original system of ordinary differential equations studied in \cite{Pritchard2014} coincide.

The program outlined in Section \ref{sec:blp} Step 2 constructs the reaction-to-reaction graph given in Figure \ref{zigzagr2r}(a), which is CS- and EM-compatible with \eqref{zigzag}.  The process outlined in Section \ref{sec:blp} Step 3 yields the following network, which is a weakly reversible, deficiency zero structural translation of \eqref{zigzag}, and is PS-, CS-, and EM-compatible with the reaction-to-reaction graph in Figure \ref{zigzagr2r}(a):
\begin{equation} \label{zigzagtranslated}
\begin{tikzcd}
X_1 + X_2 \arrow[r, yshift=0.5ex, "r_1"] & X_3 \arrow[l, yshift=-0.5ex, "r_2"] \arrow[r, yshift = 0.5ex, "r_3"] & X_3 + X_4 \arrow[l, yshift = -0.5ex, "r_4~\&~r_5"] & \\[-0.15in]
X_5 + X_6 \arrow[r, yshift=0.5ex, "r_6"] & X_7 \arrow[l, yshift=-0.5ex, "r_7"]& &\\[-0.15in]
X_9 + X_{10} + X_{11} \arrow[dddd, xshift=0.5ex,"r_{17}"] \arrow[r, yshift=0.5ex,"r_{14}"] & X_9 + X_{11} \arrow[ddddr,xshift=0.4ex,yshift=0.4ex,"r_{9}"] \arrow[r, yshift=0.5ex,"r_{13}~\&~r_{15}~\&~r_{16}"] \arrow[l,yshift=-0.5ex,"r_{10}"] & X_{11} \arrow[l, yshift=-0.5ex, "r_{8}"] & \\[-0.15in]
& & & \\[-0.15in]
& & & \\[-0.15in]
& & & \\[-0.15in]
X_9 + X_{12} \arrow[uuuu,xshift=-0.5ex,"r_{18}"] \arrow[r,"r_{19}"] & X_5 + X_9 + X_{11} \arrow[uuuu,"r_{12}"] & X_1 + X_9 +X_{11} \arrow[uuuul,xshift=-0.4ex,yshift=-0.4ex,"r_{11}"] & \\[-0.15in]
X_4 + X_{11} \arrow[r,yshift=0.5ex,"r_{20}"] & X_{13} \arrow[l,yshift=-0.5ex,"r_{21}"] & &
\end{tikzcd}
\end{equation}

\noindent In Appendix \ref{app:zigzag}, we show how the structural translation \eqref{zigzagtranslated} can be used to construct a steady state parametrization of the corresponding mass-action system \eqref{eqn:de}.


\subsection{Example: MAPK Model}
\label{sec:mapk}

Consider the following model of a mitogen-activated protein kinase (MAPK) cycle, which corresponds to {\tt biomd0000000026} in the BioModels database \cite{BioModels,M-H-K}:
\begin{equation}
\label{mapk}
\begin{tikzcd}
X+K \arrow[r,  yshift=+0.5ex, "r_1"] & XK \arrow[l,yshift=-0.5ex,"r_2"] \arrow[r,"r_3"] & X_p + K \arrow[r,yshift=+0.5ex, "r_4"] & X_pK \arrow[l,yshift=-0.5ex,"r_5"] \arrow[r,"r_6"]  & X_{pp}+K \\[-0.15in]
X_{pp}+M \arrow[r,  yshift=+0.5ex, "r_7"] & X_{pp}M \arrow[l,yshift=-0.5ex,"r_8"] \arrow[r,"r_9"] & X_p M \arrow[r,yshift=+0.5ex, "r_{10}"] & X_p + M \arrow[l,yshift=-0.5ex,"r_{11}"]  & \\[-0.15in]
X_{p}+M \arrow[r,  yshift=+0.5ex, "r_{12}"] & X^*_{p}M \arrow[l,yshift=-0.5ex,"r_{13}"] \arrow[r,"r_{14}"] & X M \arrow[r,yshift=+0.5ex, "r_{15}"] & X + M \arrow[l,yshift=-0.5ex,"r_{16}"]  & \\[-0.15in]
\end{tikzcd}
\end{equation}
The program outlined in Section \ref{sec:blp} Step 2 constructs the reaction-to-reaction graph given in Figure \ref{zigzagr2r}(b). The following weakly reversible, deficiency zero structural translation can then be constructed by the procedure outlined in Section \ref{sec:blp} Step 3:
\begin{equation}
\label{mapk-translated}
\begin{tikzcd}
X+K+M \arrow[d,xshift=+0.5ex,"r_{16}"] \arrow[r,  yshift=+0.5ex, "r_1"] & XK+M \arrow[l,yshift=-0.5ex,"r_2"] \arrow[r,"r_3"] & X_p + K +M \arrow[dr,xshift=+0.5ex,"r_{11}"] \arrow[dl,xshift=-0.5ex,"r_{12}"'] \arrow[r,yshift=+0.5ex, "r_4"] & X_pK +M \arrow[l,yshift=-0.5ex,"r_5"] \arrow[r,"r_6"]  & X_{pp}+K +M \arrow[d,xshift=+0.5ex,"r_{7}"]\\
XM + K \arrow[u,xshift=-0.5ex,"r_{15}"] & X^*_pM + K \arrow[ru,yshift=-0.5ex,"r_{13}"'] \arrow[l,"r_{14}"] & & X_pM + K \arrow[ul,yshift=-0.5ex,"r_{10}"]  & X_{pp}M+K \arrow[u,xshift=-0.5ex,"r_8"] \arrow[l,"r_9"]
\end{tikzcd}
\end{equation}

\noindent In Appendix \ref{app:mapk}, we show how the structural translation \eqref{zigzagtranslated} can be used to construct a steady state parametrization of the corresponding mass-action system \eqref{eqn:de} and guarantee the capacity for multistationarity according to Corollary 2 of \cite{C-F-M-W2016}.

\section{Conclusions}
\label{sec:conclusions}

We have presented a procedure for constructing structural translations which are weakly reversible and deficiency zero. The backbone of the algorithm is binary-linear programming (BLP) problem for determining a suitable \emph{reaction-to-reaction graph}. This graph treats the reactions of the given CRN as vertices in a new graph. We show that constructing a reaction-to-reaction graph satisfying two conditions on the edges (CS- and EM-compatibility) guarantees that a weakly reversible, deficiency zero structural translation may be constructed by imposing one further condition on the reaction-to-reaction graph (PS-compatibility). Crucially, BLP problems can be solved in polynomial time in the number of constraints by Lenstra's algorithm \cite{Lenstra1983} so that this represents a significant improvement in scalability compared to existing methods for constructing weakly reversible, deficiency zero translations.

This work presents several avenues for future work.
\begin{enumerate}
\item
The procedure outlined in Section \ref{sec:blp} is only able to produce weakly reversible, deficiency zero structural translations, which corresponds to translating all stoichiometric generators in the set of elementary flux modes into cyclic generators. Applications exist, however, for translations which are not necessarily weakly reversible or deficiency zero (e.g. absolute concentration robustness, \cite{Tonello2017,Sh-F}). Future work will focus on adapting the procedure outlined in Section \ref{sec:blp} to account for CRNs where some stoichiometric generators are not translated into cyclic generators.
\item
Recent results of \cite{J-M-P} give sufficient conditions for the parametrization of the steady state set of a generalized chemical reaction network which is weakly reversible and has a structural deficiency of zero (this is called the effective deficiency in \cite{J-M-P}). Other recent results have established conditions for mono and multistationarity within stoichiometric compatibilities \cite{C-F-M-W2016}. Integrating the structural translation procedure introduced in Section \ref{sec:blp} into a unified program for applying the results of \cite{J-M-P} and \cite{C-F-M-W2016} is ongoing. In Appendix \ref{app:a}, we outline the steps involved in this approach on the examples contained in Section \ref{sec:zigzag} and \ref{sec:mapk}.
\end{enumerate}

\noindent \textbf{Acknowledgments:}  MDJ was supported by the Henry Woodward Fund. EB was supported by the Office of Research and College of Science of San Jos\'{e} State University.


\begin{thebibliography}{10}

\bibitem{Alon2007}
Uri Alon.
\newblock {\em An introduction to systems biology: design principles of
  biological circuits}.
\newblock Chapman \& Hall/CRC, 2007.

\bibitem{C1}
Bruce~L. Clarke.
\newblock Stability of complex reaction networks.
\newblock {\em Advances in Chemical Physics}, 43:1--215, 1980.

\bibitem{C2}
Bruce~L. Clarke.
\newblock Stoichiometric network analysis.
\newblock {\em Cell. Biophys.}, 12:237--253, 1988.

\bibitem{C-F-M-W2016}
Carsten Conradi, Elisenda Feliu, Maya Mincheva, and Carsten Wiuf.
\newblock Identifying parameter regions for multistationarity.
\newblock {\em PLoS Comput. Biol.}, 13(10):e1005751, 2016.

\bibitem{C-Sh}
Carsten Conradi and Anne Shiu.
\newblock A global convergence result for processive multisite phosphorylation
  systems.
\newblock {\em Bull. Math. Biol.}, 77(1):126--155, 2015.

\bibitem{C-D-S-S}
Gheorghe Craciun, Alicia Dickenstein, Anne Shiu, and Bernd Sturmfels.
\newblock Toric dynamical systems.
\newblock {\em J. Symbolic Comput.}, 44(11):1551--1565, 2009.

\bibitem{D-M}
Alicia Dickenstein and Mercedes~P\'{e}rez Mill\'{a}n.
\newblock How far is complex balancing from detailed balancing?
\newblock {\em Bull. Math. Biol.}, 73:811--828, 2011.

\bibitem{D-M2018}
Alicia Dickenstein and Mercedes~P\'{e}rez Mill\'{a}n.
\newblock The structure of {MESSI} systems.
\newblock {\em SIAM J. Appl. Dyn. Syst.}, 17(2):1650--1682, 2018.

\bibitem{Feinberg1979}
Martin Feinberg.
\newblock Lectures on chemical reaction networks.
\newblock Unpublished written versions of lectures given at the Mathematics
  Research Center, University of Wisconsin, 1979. Available at: \url{https://crnt.osu.edu/LecturesOnReactionNetworks}

\bibitem{Feinberg1972}
Martin Feinberg.
\newblock Complex balancing in general kinetic systems.
\newblock {\em Arch. Ration. Mech. Anal.}, 49:187--194, 1972.

\bibitem{Feinberg1987}
Martin Feinberg.
\newblock Chemical reaction network structure and the stability of complex
  isothermal reactors: {I.} The deficiency zero and deficiency one theorems.
\newblock {\em Chem. Eng. Sci.}, 42(10):2229--2268, 1987.

\bibitem{Feinberg1988}
Martin Feinberg.
\newblock Chemical reaction network structure and the stability of complex
  isothermal reactors: {II.} Multiple steady states for networks of deficiency
  one.
\newblock {\em Chem. Eng. Sci.}, 43(1):1--25, 1988.

\bibitem{Feinberg1989}
Martin Feinberg.
\newblock Necessary and sufficient conditions for detailed balancing in mass
  action systems of arbitrary complexity.
\newblock {\em Chem. Eng. Sci.}, 44(9):1819--1827, 1989.

\bibitem{Feinberg1995-1}
Martin Feinberg.
\newblock The existence and uniqueness of steady states for a class of chemical
  reaction networks.
\newblock {\em Arch. Ration. Mech. Anal.}, 132:311--370, 1995.

\bibitem{Feinberg1995-2}
Martin Feinberg.
\newblock Multiple steady states for chemical reaction networks of deficiency
  one.
\newblock {\em Arch. Rational Mech. Anal.}, 132:371--406, 1995.

\bibitem{G-W1}
Cato~M. Guldberg and Peter Waage.
\newblock Studies concerning affinity.
\newblock {\em C. M. Forhandlinger: Videnskabs-Selskabet i Chistiana}, page~35,
  1864.

\bibitem{Hi}
Archibald Hill.
\newblock The possible effects of the aggregation of the molecules of
  haemoglobin on its dissociation curves.
\newblock {\em J. Physiol.}, 40(4), 2010.

\bibitem{H}
Fritz Horn.
\newblock Necessary and sufficient conditions for complex balancing in chemical
  kinetics.
\newblock {\em Arch. Ration. Mech. Anal.}, 49:172--186, 1972.

\bibitem{H-J1}
Fritz Horn and Roy Jackson.
\newblock General mass action kinetics.
\newblock {\em Arch. Ration. Mech. Anal.}, 47:81--116, 1972.

\bibitem{Ingalls}
Brian~P. Ingalls.
\newblock {\em Mathematical Modeling in Systems Biology: An Introduction}.
\newblock MIT Press, 2013.

\bibitem{J1}
Matthew~D. Johnston.
\newblock Translated chemical reaction networks.
\newblock {\em Bull. Math. Biol.}, 76(5):1081--1116, 2014.

\bibitem{J2}
Matthew~D. Johnston.
\newblock A computational approach to steady state correspondence of regular
  and generalized mass action systems.
\newblock {\em Bull. Math. Biol.}, 77(6):1065--1100, 2015.

\bibitem{J-M-P}
Matthew~D. Johnston, Stefan M\"{u}ller, and Casian Pantea.
\newblock Rational parametrizations of steady state manifolds for a class of
  mass-action systems.
\newblock Submitted, 2018. Available on the ArXiv at arXiv:1805.09295.


\bibitem{Jones2006}
J.D.G. Jones and J.L. Dangl.
\newblock The plant immune system.
\newblock {\em Nature}, 444:323--329, 2006.

\bibitem{Lenstra1983}
H.~W.~Lenstra.
\newblock Integer programming with a fixed number of variables.
\newblock {\em Math. Oper. Res.}, 8:538--548, 1983.

\bibitem{BioModels}
C.~Li, M.~Donizello, N.~Rodriguez, H.~Dharuri, L.~Endler, V.~Chelliah, L.~Li,
  E.~He, A.~Henry, M.I. Stefan, J.L. Snoep, M.~Hucka, N.~Le Lovere, and
  C.~Laibe.
\newblock Bio{M}odels {D}atabase: {A}n enhance, curated and annotated resource
  for published quantitative kinetic models.
\newblock {\em BMC Syst. Biol.}, 4:92, 2010.

\bibitem{M-H-K}
Nick~I. Markevich, Jan~B. Hoek, and Boris~N. Kholodenko.
\newblock Signaling switches and bistability arising from multisite
  phosphorylation in protein kinase cascades.
\newblock {\em J. Cell. Biol.}, 164(3):353--359, 2004.

\bibitem{M-M}
Leonor Michaelis and Maud Menten.
\newblock Die kinetik der invertinwirkung.
\newblock {\em Biochem. Z.}, 49:333--369, 1913.

\bibitem{M-D-S-C}
Mercedes~P\'{e}rez Mill\'{a}n, Alicia Dickenstein, Anne Shiu, and Carsten
  Conradi.
\newblock Chemical reaction systems with toric steady states.
\newblock {\em Bull. Math. Biol.}, 74(5):1027--1065, 2012.

\bibitem{M-F-R-C-S-D}
Stefan M\"{u}ller, Elisenda Feliu, George Regensburger, Carsten Conradi, Anne
  Shiu, and Alicia Dickenstein.
\newblock Sign conditions for injectivity of generalized polynomial maps with
  applications to chemical reaction networks and real algebraic geometry.
\newblock {\em Found. Comput. Math.}, 16(1):69--97, 2016.

\bibitem{MR2012}
Stefan M\"{u}ller and Georg Regensburger.
\newblock Generalized mass action systems: Complex balancing equilibria and
  sign vectors of the stoichiometric and kinetic-order subspaces.
\newblock {\em SIAM J. Appl. Math.}, 72(6):1926--1947, 2012.

\bibitem{MR2014}
Stefan M\"{u}ller and Georg Regensburger.
\newblock Generalized mass-action systems and positive solutions of polynomial
  equations with real and symbolic exponents (invited talk).
\newblock In Gerdt V.P., Koepf W., Seiler W.M., and Vorozhtsov E.V., editors,
  {\em Computer Algebra in Scientific Computing. CASC 2014. Lecture Notes in
  Computer Science}, 8660:302--323, Springer, 2014.

\bibitem{Orth2010}
Jeffrey~D. Orth, Ines Thiele, and Bernard~O. Palsson.
\newblock What is flux balance analysis?
\newblock {\em 	Nat. Biotechnol.}, 28:245--248, 2010.

\bibitem{Pritchard2014}
Leighton Pritchard and Paul~R.J. Birch.
\newblock The zigzag model of plant-microbe interactions: is it time to move
  on?
\newblock {\em Mol. Plant Pathol.}, 15(9):865--870, 2014.

\bibitem{Sh-F}
Guy Shinar and Martin Feinberg.
\newblock Structural sources of robustness in biochemical reaction networks.
\newblock {\em Science}, 327(5971):1389--1391, 2010.

\bibitem{Tonello2016crnpy}
Elisa Tonello.
\newblock {C}rn{P}y: a python library for the analysis of chemical reaction
  networks.
\newblock 2016. Available at: \url{https://github.com/etonello/crnpy}

\bibitem{Tonello2017}
Elisa Tonello and Matthew~D. Johnston.
\newblock Network translation and steady state properties of chemical reaction
  systems.
\newblock {\em Bull. Math. Biol.}, 80(9):2306--2337, 2018.

\bibitem{Wiback2002}
Sharon~J. Wiback and Bernard~O. Palsson.
\newblock Extreme pathway analysis of human red blood cell metabolism.
\newblock {\em Biophys. J.}, 83(2):808--818, 2002.

\bibitem{Zang2013}
J\"{u}rden Zanghellini, David~E. Ruckerbauer, Michael Hanscho, and Christian
  Jungreuthmayer.
\newblock Elementary flux modes in a nutshell: Properties, calculation and
  applications.
\newblock {\em Biotechnol. J.}, 8(9):1009--1016, 2013.

\end{thebibliography}

\appendix

\section{Appendix - Parametrization Method}
\label{app:a}

While two structural translations have the same stoichiometric matrices $\Gamma$ and $\Gamma'$, they may nevertheless have different mass-action systems \eqref{eqn:de} due to differences in $R(\mathbf{x})$. In this Appendix, we outline the method by which a steady state parametrization may be constructed from a structural parametrization as constructed by the algorithm presented in Section \ref{sec:blp}. 

For ease of notation and continuity, rather than repeat the technical definitions and Theorems of \cite{MR2014} and \cite{J-M-P}, we outline the parametrization procedure through examples.

\subsection{Histidine Kinase Model}
\label{app:parametrization}

We use the histidine kinase network in Figure \ref{figure1} (left) as a motivating example. Through application of the algorithm presented in Section \ref{sec:blp}, we were able to correspond the following CRN (left) with the indicated structural translation (right):
\[
\begin{sideways}\hspace{-0.36in}\mbox{\textbf{Network 1}}\end{sideways} \; \left\{ \; \;\begin{tikzcd}
X \arrow[r, "r_1"] & X_p \\[-0.15in]
X_p + Y \arrow[r,  yshift=+0.5ex,"r_2"] & X + Y_p \arrow[l,  yshift=-0.5ex,"r_3"]  \\[-0.15in]  
Y_p \arrow[r, "r_4"] & Y
\end{tikzcd} \right.
\; \; \; \; \; \Longleftrightarrow \; \; \; \; \;
\begin{sideways}\hspace{-0.36in}\mbox{\textbf{Network 2}}\end{sideways} \; \left\{ \; \; \begin{tikzcd}
X+Y \arrow[rr,yshift=+0.5ex,"r_1"] & & 
X_p+Y \arrow[dl,xshift=1ex,"r_2"]\\
 & X+Y_p \arrow[ur,xshift=-1ex,"r_3"] \arrow[ul,"r_4"] &
\end{tikzcd} \right.
\]

Although these two networks have the same reaction vectors (i.e. $\Gamma = \Gamma'$), the dynamical equations \eqref{eqn:de} do not coincide due to differences in $R(\mathbf{x})$. Specifically, the source complex of $r_1$ and $r_4$ differ in Network 1 from Network 2. To accommodate this difference, we map the source complexes from Network 1 into a secondary set of complexes known as kinetic-order complexes in Network 2. We can represent this with the following network:
\begin{equation} \label{examplegcrn2}
\begin{tikzcd}
\ovalbox{$\begin{array}{c}  1 \\ \\ \end{array}\Bigg\lvert \begin{array}{c} X+Y \\ (X) \end{array}$} \arrow[r,yshift=+0.5ex,"r_1"] & 
\ovalbox{$\begin{array}{c}  2 \\ \\ \end{array}\Bigg\lvert\begin{array}{c} X_p+Y \\ (X_p+Y) \end{array}$} \arrow[d,xshift=0.5ex,"r_2"]\\
\ovalbox{$\begin{array}{c}  4 \\ \dag \\ \end{array}\Bigg\lvert\begin{array}{c} X+Y_p \\ (Y_p) \end{array}$} \arrow[u,"r_4"] & 
\ovalbox{$\begin{array}{c}  3 \\ \dag \\ \end{array}\Bigg\lvert\begin{array}{c} X+Y_p \\ (X+Y_p) \end{array}$} \arrow[u,xshift=-0.5ex,"r_3"]
\end{tikzcd}
\end{equation}

The network \eqref{examplegcrn2} is an example of a \emph{generalized chemical reaction network} (GCRN) \cite{MR2012,MR2014}. In a GCRN, each vertex is assigned two complexes: a \emph{stoichiometric complex} (unbracketed) and a \emph{kinetic-order complex} (bracketed). In the corresponding \emph{generalized mass-action system}
\begin{equation}
\label{eqn:de2}
\frac{d\mathbf{x}}{dt} = \Gamma \tilde R(\mathbf{x})
\end{equation}
the reaction vectors forming $\Gamma$ are determined by the differences of the stoichiometric complexes while the monomials in $\tilde R(\mathbf{x})$ are determined by the kinetic-order complexes. Denote the $i$th kinetic-order complex by $\tilde y_i$, we have that $\tilde R(\mathbf{x})$ has entries $\tilde R_i(\mathbf{x}) = \prod_{j=1}^m x_j^{[\tilde y_{s(i)}]_j}$. For example, the term in $\tilde R(\mathbf{x})$ corresponding to $r_1$ in \eqref{examplegcrn2} is $k_1 x$ rather than the stoichiometrically-determined $k_1 x y$. It can be easily checked that the dynamical equations \eqref{eqn:de2} corresponding to \eqref{examplegcrn2} coincide with the dynamical equations \eqref{eqn:de} corresponding to Network 1.

Note that when converting from Network 2 to \eqref{examplegcrn2}, we split the vertex $X+Y_p$. Consequently, the stoichiometric complex $X+Y_p$ is repeated at vertexes 3 and 4 in \eqref{examplegcrn2} (indicted with $\dag$). This is allowed by \cite{MR2014,J-M-P} and, in fact, required since $r_3$ and $r_4$ would otherwise have multiple kinetic-order complexes at a single vertex (although, with some supplemental conditions, this is allowed in \cite{J2,Tonello2017}). 
In order to regain weak reversibility, the authors of \cite{J-M-P} introduce a new set of edges (called ``phantom reactions") which connect stoichiometrically identical complexes. Notice that introducing such reactions introduces zero columns in $\Gamma$ and therefore does not alter the corresponding dynamical equations \eqref{eqn:de2}.

For technical reasons, the authors of \cite{J-M-P} imposed further rules upon the splitting of stoichiometric complexes and the introduction of phantom reactions. They define equivalence classes of stoichiometrically identical complexes and select from within each such class a distinguished vertex (indicated with a $\star$). The set of phantom reactions is then introduced such that: 
\begin{enumerate}
\item
All ``true reactions'' (i.e from the set $\R$) which have their product at any vertex in this equivalence class have the distinguished vertex as its product.
\item
The phantom reactions between vertexes on this equivalence class consist only of reactions with the distinguished complex as its source and the remaining complexes as the product.
\end{enumerate}
We may interpret the distinguished vertices as hubs through which are all paths through an equivalence class of stoichiometrically identical complexes must pass. Such a construction produces a $V^{\star}$-directed GCRN which is important in the construction of positive parametrizations \cite{J-M-P}.

For the network \eqref{examplegcrn2}, we select vertex 3 as the distinguish vertex (indicated with $\star$) and label the phantom edge with a free parameter $\sigma$:
\begin{equation} \label{examplegcrn}
\begin{tikzcd}
\ovalbox{$\begin{array}{c}  1 \\ \\ \end{array}\Bigg\lvert \begin{array}{c} X+Y \\ (X) \end{array}$} \arrow[r,yshift=+0.5ex,"r_1"] & 
\ovalbox{$\begin{array}{c}  2 \\ \\ \end{array}\Bigg\lvert \begin{array}{c} X_p+Y \\ (X_p+Y) \end{array}$} \arrow[d,xshift=0.5ex,"r_2"]\\
\ovalbox{$\begin{array}{c}  4 \\ \dag \\ \end{array}\Bigg\lvert \begin{array}{c} X+Y_p \\ (Y_p) \end{array}$} \arrow[u,"r_4"] & 
\ovalbox{$\begin{array}{c}  3 \star \\ \dag\\ \end{array}\Bigg\lvert \begin{array}{c} X+Y_p \\ (X+Y_p) \end{array}$} \arrow[u,xshift=-0.5ex,"r_3"] \arrow[l,red,"\sigma"]
\end{tikzcd}
\end{equation}
Notice that only $r_2$ has a product in the equivalence class of vertexes $\{3, 4\}$ and its product is the distinguished complex $3$ (condition 1), and the only reaction on vertexes $\{3, 4 \}$ goes from the distinguished vertex $3$ to the remaining vertex $4$ (condition 2). The GCRN \eqref{examplegcrn} is therefore $V^{\star}$-directed.

In \cite{J-M-P}, the authors show that, if the deficiency of the structural translation (called the \emph{effective deficiency} in \cite{J-M-P}) is zero and the corresponding GCRN is $V^{\star}$-directed, then the positive steady state set of the original dynamical system \eqref{eqn:de} can be characterized by the \emph{complex-balanced} steady states of the dynamical system \eqref{eqn:de2}, namely, the equation
\begin{equation}
\label{complexbalanced}
A_k \tilde R(\mathbf{x}) = 0
\end{equation}
where $A_k \in \mathbb{R}^{n \times n}$ is the Laplacian of the reaction graph of the GCRN. For the network \eqref{examplegcrn}, this corresponds to the system:
\[\left[ \begin{array}{cccc} -k_1 & 0 & 0 & k_4 \\ k_1 & - k_2 & k_3 & 0 \\ 0 & k_2 & -k_3 - \sigma & 0 \\ 0 & 0 & \sigma & - k_4 \end{array} \right] \left[ \begin{array}{c} x \\ x_py \\ xy_p \\ y_p \end{array} \right] = \left[ \begin{array}{c} 0 \\ 0 \\ 0 \\ 0 \end{array} \right]\]
Relationships between $\mbox{ker}(A_k)$ and the steady state set of mass-action systems has been studied extensively in recent years. It is known that, for weakly reversible networks, $\mbox{ker}(A_k)$ can be characterized by algebraic combinations of the rate constants of a network known as ``tree constants'' \cite{J1,J-M-P} which we summarize in Appendix \ref{app:procedure}. For this network, we can directly compute that $\mbox{ker}(\tilde A_k) = \mbox{span} \{ K_1, K_2, K_3, K_4\}$ where
\[\begin{aligned}
K_1 & = k_2k_4\sigma,&
K_2 & = k_1(k_3 + \sigma)k_4,&
K_3 & = k_1k_2k_4, &
K_4 & = k_1k_2 \sigma
\end{aligned}\]
are the tree constants. The steady state condition $(x, x_py,xy_p,y_p) \in \mbox{ker}(\tilde A_k)$ gives the implicit equations
\[\frac{x}{k_2k_4\sigma} = \frac{x_py}{k_1(k_3+\sigma)k_4} = \frac{xy_p}{k_1k_2k_4} = \frac{y_p}{k_1k_2\sigma}.\]
Taking pairwise differences, this gives the following log-linear system of equations:
\begin{equation}
\label{loglinear}
\left[ \begin{array}{cccc} -1 & 1 & 1 & 0 \\ 0 & 0 & 0 & 1 \\ -1 & 0 & 0 & 1 \end{array} \right] \left[ \begin{array}{c} \ln(x) \\ \ln(x_p) \\ \ln(y) \\ \ln(y_p) \end{array} \right] = \left[ \begin{array}{c} \frac{k_1(k_3 + \sigma)}{k_2\sigma} \\ \frac{k_1}{\sigma} \\ \frac{k_1}{k_4} \end{array} \right]\end{equation}
Surprisingly, the solvability of the system \eqref{loglinear} depends on the deficiency of network \eqref{examplegcrn} taken with only the kinetic-order complexes:
\begin{equation} \label{kinetic}
\begin{tikzcd}
X \arrow[r,yshift=+0.5ex,"r_1"] & 
X_p+Y \arrow[d,xshift=0.5ex,"r_2"]\\
Y_p \arrow[u,"r_4"] & 
X+Y_p \arrow[u,xshift=-0.5ex,"r_3"] \arrow[l,"\sigma"]
\end{tikzcd}
\end{equation}
The deficiency of the network \eqref{kinetic} is known as the \emph{kinetic-order deficiency} \cite{MR2012,MR2014,J-M-P}. We can compute that the deficiency of \eqref{kinetic} is zero so that the kinetic-order deficiency of \eqref{examplegcrn} is zero. Consequently, the log-linear system \eqref{loglinear} is guaranteed to be consistent and therefore have a solution for all values of rate constants (including $\sigma$) \cite{J-M-P}. For this example, \eqref{loglinear} can be solved for the log concentrations, which can then be exponentiated to give the following parametrization:
\[
\left \{ \quad
\begin{aligned}
x &= \frac{k_4}{\sigma} , & y &= \tau , \\ 
x_p &= \frac{k_1(k_3+\sigma)k_4}{k_2\sigma^2\tau}, \; \; \; \; \; & y_p &= \frac{k_1}{\sigma},
\end{aligned}
\right.
\]
where $\sigma, \tau > 0$ are positive parameters. Notice that the parameter $\tau$ has arisen from parametrizing the nullspace of the coefficient matrix in \eqref{loglinear}, which is the span of the vector $(0,-1,1,0)$.

\subsection{General Procedure for Parametrizations}
\label{app:procedure}

For a given GCRN, we let $\tilde{y}_i$ denote the kinetic-order complex at the vertex labeled $i$ and define $\mathcal{T} \subseteq \R$ to be the set of all trees which span the linkage class containing the vertex $i$. The tree constants $K_i$ corresponding to the vertex labeled $i$ is given by
\begin{equation}
\label{treeconstant}
K_i = \sum_{T \in \mathcal{T}} \prod_{r_i \in T} k_i.
\end{equation}
By Lemma 12 of \cite{J-M-P}, if the GCRN has a structural deficiency of zero, we have the following representations of the steady state set of the corresponding generalized mass-action system:
\begin{equation}
\label{implicit}
\frac{\mathbf{x}^{\tilde y_i}}{K_i} = \frac{\mathbf{x}^{\tilde y_j}}{K_j} \; \; \; \Longleftrightarrow \; \; \; (\tilde y_j - \tilde y_i)^T \ln(\mathbf{x}) = \ln\left( \frac{K_j}{K_i} \right)
\end{equation}
for all vertices $i$ and $j$ belonging to the same linkage class. We can use the log-linear equation on the right of \eqref{implicit} to construct a linear system in the log concentrations. We define a matrix $M$ such columns $M_{\cdot, k} = \tilde y_j - \tilde{y_i}$ and a vector $b$ with entries $b_k = \ln(K_j/K_i)$ where the pairs $(i,j)$ are chosen to be a maximal set of such that the resulting set spans the vertices of the underlying GCRN and does not have any nontrivial cycles. This process produces the following log-linear system
\begin{equation}
\label{loglinear2}
M^T \ln(\mathbf{x}) = b.
\end{equation}
A structural deficiency of zero guarantees all steady states can be found by solving \eqref{loglinear2} (Lemma 12, \cite{J-M-P}). A kinetic-order deficiency of zero guarantees the solvability of this system for all values of the rate constants (Theorem 14 part 1, \cite{J-M-P}). A GCRN with a nonzero kinetic-order deficiency, however, may still produce a solvable system \eqref{loglinear2} provided certain supplemental conditions on the rate parameters are satisfied (Theorem 14 part 2, \cite{J-M-P}).

The example in Appendix \ref{app:parametrization} suggests the following general procedure for determining a positive steady state parametrization for mass-action systems \eqref{eqn:de}:
\begin{enumerate}
\item[]
\textbf{Step 1:} Construct a weakly reversible, deficiency zero structural translation by the algorithm presented in Section \ref{sec:blp}.
\item[]
\textbf{Step 2:} Transfer source complexes from the original CRN as kinetic-order complexes in the network GCRN, splitting stoichiometric complexes as necessary.
\item[]
\textbf{Step 3:} Within each equivalence class of stoichiometrically identical complexes, select distinguished vertices and phantom edges so that the resulting GCRN is $V^{\star}$-directed. Note that by \cite{J-M-P} the choice of distinguished vertices may be made arbitrarily.
\item[]
\textbf{Step 4:} Compute the \emph{kinetic-order deficiency}. (The deficiency of the network with only the kinetic-order complexes from the $V^{\star}$-directed network found in Step 3.) If the kinetic-order deficiency is zero, skip to Step 5; otherwise proceed to Step 4*.
\item[]
\textbf{Step 4*:} Determine a basis $\{ c_1, \ldots, c_{\tilde \delta} \}$ of $\mbox{ker}(M)$ and for every vector $c_i$ attempt to solve the system $c_i^T b = 0$ for the phantom edge parameters $\sigma_j$. If these conditions cannot be satisfied, the procedure fails. Otherwise, substitute the solved parameters $\sigma_j$ into the GCRN constructed in Step 3 and proceed to Step 5.
\item[]
\textbf{Step 5:} Compute the ``tree constants'' at each vertex of this $V^{\star}$-directed GCRN.
\item[]
\textbf{Step 6:} Set up and solve the log-linear system \eqref{loglinear2} for the concentrations.
\end{enumerate}


\subsection{ZigZag Model Example}
\label{app:zigzag}

Reconsider the zigzag model of plant-pathogen interactions \eqref{zigzag}. We now outline how the steps described in Appendix \ref{app:procedure} apply to this network.

\paragraph*{Step 1:} We were able to use the algorithm described in Section \ref{sec:blp} to determine the following structural translation:
\begin{equation} \label{zigzagtranslated2}
\begin{tikzcd}
X_1 + X_2 \arrow[r, yshift=0.5ex, "r_1"] & X_3 \arrow[l, yshift=-0.5ex, "r_2"] \arrow[r, yshift = 0.5ex, "r_3"] & X_3 + X_4 \arrow[l, yshift = -0.5ex, "r_4~\&~r_5"] & \\[-0.15in]
X_5 + X_6 \arrow[r, yshift=0.5ex, "r_6"] & X_7 \arrow[l, yshift=-0.5ex, "r_7"]& &\\[-0.15in]
X_9 + X_{10} + X_{11} \arrow[dddd, xshift=0.5ex,"r_{17}"] \arrow[r, yshift=0.5ex,"r_{14}"] & X_9 + X_{11} \arrow[ddddr,xshift=0.4ex,yshift=0.4ex,"r_{9}"] \arrow[r, yshift=0.5ex,"r_{13}~\&~r_{15}~\&~r_{16}"] \arrow[l,yshift=-0.5ex,"r_{10}"] & X_{11} \arrow[l, yshift=-0.5ex, "r_{8}"] & \\[-0.15in]
& & & \\[-0.15in]
& & & \\[-0.15in]
& & & \\[-0.15in]
X_9 + X_{12} \arrow[uuuu,xshift=-0.5ex,"r_{18}"] \arrow[r,"r_{19}"] & X_5 + X_9 + X_{11} \arrow[uuuu,"r_{12}"] & X_1 + X_9 +X_{11} \arrow[uuuul,xshift=-0.4ex,yshift=-0.4ex,"r_{11}"] & \\[-0.15in]
X_4 + X_{11} \arrow[r,yshift=0.5ex,"r_{20}"] & X_{12} \arrow[l,yshift=-0.5ex,"r_{21}"] & &
\end{tikzcd}
\end{equation}
As expected by the algorithm, this network is weakly reversible and deficiency zero. It follows from Lemma 12 of \cite{J-M-P} that all of the steady states can be found by setting up and solving the log-linear system \eqref{loglinear2}.

\paragraph*{Steps 2 \& 3:} Notice that the complexes $X_3 + X_4$, $X_9 + X_{10} + X_{11}$, and $X_9 + X_{11}$ have multiple source complexes which are translated to them from \eqref{zigzag}. 
We therefore split these vertices in \eqref{zigzagtranslated2} when assigning kinetic-order complexes.  
We also need to select distinguish complexes and add phantom edges to satisfy the conditions of being $V^\star$-directed given in Appendix \ref{app:parametrization}. This can be accomplished by the following network, where the phantom edges are labeled with $\sigma_i$, $i=1, \ldots, 4$, the equivalence classes of stoichiometrically identical complexes are indicated with the symbols $\mathsection$, $\dag$, and $\ddag$, and the distinguished vertices are indicated with $\star$.
\begin{equation}\label{mess3}\footnotesize
\begin{tikzcd}[column sep=0.5cm]
\mbox{\ovalbox{$\begin{array}{c}  1 \\ \\ \end{array}\Bigg\lvert \begin{array}{c} X_1 + X_2\\ (X_1 + X_2) \end{array}$}} \arrow[r,yshift=0.5ex,"r_1"] & \mbox{\ovalbox{$\begin{array}{c}  2 \\ \\ \end{array}\Bigg\lvert \begin{array}{c} X_3\\ (X_3) \end{array}$}} \arrow[l,yshift=-0.5ex,"r_2"] \arrow[r,yshift=0.5ex,"r_3"] 
& \mbox{\ovalbox{$\begin{array}{c}  3\star \\ \mathsection \\ \end{array}\Bigg\lvert \begin{array}{c} X_3 + X_4\\ (X_4) \end{array}$}} \arrow[l,yshift=-0.5ex,"r_4"] \arrow[d,red,"\sigma_1"]  \\
\mbox{\ovalbox{$\begin{array}{c}  5 \\ \\ \end{array}\Bigg\lvert \begin{array}{c} X_5 + X_6\\ (X_5 + X_6) \end{array}$}} \arrow[r,yshift=0.5ex,"r_6"] & \mbox{\ovalbox{$\begin{array}{c}  6 \\ \\ \end{array}\Bigg\lvert \begin{array}{c} X_7\\ (X_7) \end{array}$}} \arrow[l,yshift=-0.5ex,"r_7"] & \mbox{\ovalbox{$\begin{array}{c}  4 \\ \mathsection \\ \end{array}\Bigg\lvert \begin{array}{c} X_3 + X_4\\ (X_4 + X_5) \end{array}$}} \arrow[lu,"r_5"]\\
\mbox{\ovalbox{$\begin{array}{c} 7 \\ \dag \\ \end{array}\Bigg\lvert \begin{array}{c} X_9 + X_{10} + X_{11} \\ (X_{10}) \end{array}$}} \arrow[dr,"r_{14}"] & \mbox{\ovalbox{$\begin{array}{c}  8 \\ \ddag \\ \end{array}\Bigg\lvert \begin{array}{c} X_9 + X_{11}\\ (X_7 + X_9) \end{array}$}} \arrow[dr,"r_{15}"] & \mbox{\ovalbox{$\begin{array}{c} 9 \\ \ddag \\ \end{array}\Bigg\lvert \begin{array}{c} X_9 + X_{11}\\ (X_4 + X_9) \end{array}$}} \arrow[d,"r_{16}"] \\
\mbox{\ovalbox{$\begin{array}{c} 10 \star \\ \dag \\ \end{array}\Bigg\lvert \begin{array}{c} X_9 + X_{10} + X_{11} \\ (X_{10} + X_{11}) \end{array}$}} \arrow[u,red,"\sigma_2"] \arrow[d,xshift=0.5ex,"r_{17}"]  & \mbox{\ovalbox{$\begin{array}{c}  11 \star \\ \ddag \\ \end{array}\Bigg\lvert \begin{array}{c} X_9 + X_{11}\\ (X_9) \end{array}$}} \arrow[l,"r_{10}"] \arrow[u,red,"\sigma_3"] \arrow[rd,xshift=0.4ex,yshift=0.4ex,"r_{9}"] \arrow[r,yshift=0.5ex,"r_{13}"] \arrow[ru,red,"\sigma_4" pos=0.8,bend left=24] & \mbox{\ovalbox{$\begin{array}{c}  12 \\ \\ \end{array}\Bigg\lvert \begin{array}{c} X_{11} \\ (X_8) \end{array}$}} \arrow[l,yshift=-0.5ex,"r_{8}"]\\
\mbox{\ovalbox{$\begin{array}{c}  13 \\ \\ \end{array}\Bigg\lvert \begin{array}{c} X_9 + X_{12}\\ (X_{12}) \end{array}$}} \arrow[u,xshift=-0.5ex,"r_{18}"] \arrow[r,"r_{19}"] & \mbox{\ovalbox{$\begin{array}{c}  14 \\ \\ \end{array}\Bigg\lvert \begin{array}{c} X_5 + X_9 + X_{11}\\ (X_5) \end{array}$}} \arrow[u,"r_{12}"] & \mbox{\ovalbox{$\begin{array}{c}  15 \\ \\ \end{array}\Bigg\lvert \begin{array}{c} X_1 + X_9 + X_{11}\\ (X_1) \end{array}$}} \arrow[lu,xshift=-0.4ex,yshift=-0.4ex,"r_{11}"]\\
\mbox{\ovalbox{$\begin{array}{c}  16 \\ \\ \end{array}\Bigg\lvert \begin{array}{c} X_4 + X_{11}\\ (X_4 + X_{11}) \end{array}$}} \arrow[r,yshift=0.5ex,"r_{20}"] & \mbox{\ovalbox{$\begin{array}{c}  17 \\ \\ \end{array}\Bigg\lvert \begin{array}{c} X_{13}\\ (X_{13}) \end{array}$}} \arrow[l,yshift=-0.5ex,"r_{21}"]&
\end{tikzcd}
\end{equation}

\paragraph*{Step 4:} The kinetic-order deficiency is the deficiency of the CRN produced by considering only the kinetic-order (bracketed) complexes in \eqref{mess3}. 
It can be quickly computed that the deficiency is $\delta = n - \ell - s = 17-4-13 = 0$. It follows from Theorem 14 of \cite{J-M-P} that the remainder of the steps may be performed to yield a steady state parametrization.

\paragraph*{Step 5:} From \eqref{mess3}, we compute the following tree constants:
\[ \small \begin{aligned}
K_1 & = k_2k_5(k_4+\sigma_1)&K_{10} & = k_8k_{10}k_{11}k_{12}k_{14}k_{15}k_{16}(k_{18}+k_{19})\\
K_2 & = k_1(k_4+\sigma_1)k_5&K_{11} & = k_8k_{11}k_{12}k_{14}k_{15}k_{16}(k_{17}k_{19}+(k_{18}+k_{19})\sigma_2)\\
K_3 & = k_1k_3k_5&K_{12} & = k_{11}k_{12}(\sigma_3+\sigma_4+k_{13})k_{14}k_{15}k_{16}((\sigma_2+k_{17})k_{19}+k_{18}\sigma_2)\\
K_4 & = k_1k_3\sigma_1&K_{13} & = k_8k_{10}k_{11}k_{12}k_{14}k_{15}k_{16}k_{17}\\
K_5 & = k_7&K_{14} & = k_8k_{10}k_{11}k_{14}k_{15}k_{16}k_{17}k_{19}\\
K_6 & = k_6&K_{15} & = k_8k_9k_{12}k_{14}k_{15}k_{16}(k_{17}k_{19}+(k_{18}+k_{19})\sigma_2)\\
K_7 & = k_8k_{10}k_{11}k_{12}k_{15}k_{16}(k_{18}+k_{19})\sigma_2&K_{16} & = k_{21}\\
K_8 & = k_8k_{11}k_{12}k_{14}k_{16}((\sigma_2+k_{17})k_{19}+k_{18}\sigma_2)\sigma_3&K_{17} & = k_{20}\\
K_9 & = k_8k_{11}k_{12}k_{14}k_{15}((\sigma_2+k_{17})k_{19}+k_{18}\sigma_2)\sigma_4\\
\end{aligned} \]

\paragraph*{Step 6:} The log-linear system \eqref{loglinear2} can be set-up for any maximal set of pairs of vertices lying in the same linkage class. We take the pairs
\[\small \{ 1, 2\}, \{1, 3\}, \{1, 4\}, \{5, 6\}, \{7, 8\}, \{7, 9\}, \{7, 10\}, \{7, 11\}, \{7, 12\}, \{7, 13\}, \{7, 14\}, \{7, 15\}, \{16, 17\}.\]
This gives the following linear system in the log concentrations \eqref{loglinear2}:
\[\small \left[ \begin{array}{ccccccccccccc}-1&-1&1&0&0&0&0&0&0&0&0&0&0\\-1&-1&0&1&0&0&0&0&0&0&0&0&0\\-1&-1&0&1&1&0&0&0&0&0&0&0&0\\0&0&0&0&-1&-1&1&0&0&0&0&0&0\\0&0&0&0&0&0&1&0&1&-1&0&0&0\\0&0&0&1&0&0&0&0&1&-1&0&0&0\\0&0&0&0&0&0&0&0&0&0&1&0&0\\0&0&0&0&0&0&0&0&1&-1&0&0&0\\0&0&0&0&0&0&0&1&0&-1&0&0&0\\0&0&0&0&0&0&0&0&0&-1&0&1&0\\0&0&0&0&1&0&0&0&0&-1&0&0&0\\1&0&0&0&0&0&0&0&0&-1&0&0&0\\0&0&0&-1&0&0&0&0&0&0&-1&0&1\end{array} \right] \left[ \begin{array}{c} \ln(x_1) \\ \ln(x_2) \\ \ln(x_3) \\ \ln(x_4) \\ \ln(x_5) \\ \ln(x_6) \\ \ln(x_7) \\ \ln(x_8) \\ \ln(x_9) \\ \ln(x_{10}) \\ \ln(x_{11}) \\ \ln(x_{12}) \\ \ln(x_{13}) \end{array} \right] = \left[ \begin{array}{c} \ln\left( K_2/K_1 \right) \\ \ln\left( K_3/K_1 \right) \\ \ln\left( K_4/K_1 \right) \\ \ln\left( K_6/K_5 \right) \\ \ln\left( K_8/K_7 \right) \\ \ln\left( K_9/K_7 \right) \\ \ln\left( K_{10}/K_7 \right) \\ \ln\left( K_{11}/K_7 \right) \\ \ln\left( K_{12}/K_7 \right) \\ \ln\left( K_{13}/K_7 \right) \\ \ln\left( K_{14}/K_7 \right) \\ \ln\left( K_{15}/K_7 \right) \\ \ln\left( K_{17}/K_{16} \right) \end{array} \right]\]

\noindent Since the kinetic-order deficiency is zero, this is a consistent system and therefore guaranteed to have a solution for all rate constants (Theorem 14, \cite{J-M-P}). Solving the system for $\ln(x_i)$ and then exponentiating gives the following solution, which is a rational parametrization of the steady state set of the mass-action system \eqref{eqn:de} corresponding to \eqref{zigzag} in the parameters $\sigma_1, \sigma_2, \sigma_3, \sigma_4 \in \mathbb{R}_{> 0}$:
\begin{equation}
\label{parametrization}
\small \begin{aligned}
x_1 & = \frac{ k_9 k_{12}(k_{17} k_{19}+(k_{18}+k_{19})\sigma_2)\sigma_1}{k_5k_8k_{10}k_{17}k_{19}} & x_8 & = \frac{k_{12}(k_{13}+\sigma_3+\sigma_4)((k_{17}+\sigma_2)k_{19}+k_{18}\sigma_2) \sigma_1}{ k_5k_{10}k_{11}  k_{17} k_{19}}\\
x_2 & = \frac{k_2 (k_4+\sigma_1)k_5 k_{10} k_{11}  k_{17}k_{19} \sigma_4 }{ k_1 k_3  k_9k_{12}k_{16} (k_{17} k_{19}+(k_{18} + k_{19}) \sigma_2) \sigma_1}& x_9 & = \frac{ k_{12}(k_{17} k_{19}+(k_{18}+k_{19}) \sigma_2)\sigma_1 }{k_5 k_{10} k_{17} k_{19}}\\
x_3 & = \frac{(k_4+\sigma_1)\sigma_4 }{ k_3k_{16}}& x_{10} & = \frac{k_{12} (k_{18}+k_{19})\sigma_1 \sigma_2 }{k_5 k_{14}k_{17} k_{19}}\\
x_4 & = \frac{\sigma_4}{k_{16}}&x_{11} & = \frac{k_{14}}{\sigma_2}\\
x_5 & = \frac{\sigma_1}{k_5}&x_{12} & = \frac{ k_{12}\sigma_1}{k_5 k_{19}}\\
x_6 & = \frac{ k_5 k_7\sigma_3}{k_6k_{15} \sigma_1 }&x_{13} & = \frac{ k_{14} k_{20}\sigma_4}{k_{16} k_{21}\sigma_2}\\
x_7 & = \frac{\sigma_3}{k_{15}}\\
\end{aligned}
\end{equation}

Notice that this parametrization does not guarantee that for a given initial condition $\mathbf{x}(0) \in \mathbb{R}_{\geq 0}^m$ the parametrization intersects the relevant compatibility class $(\mathbf{x}(0) + S) \cap \mathbb{R}_{\geq 0}^m$. For this example, we can observe that $x_8$ experiences no stoichiometric change in any of the system's interactions and therefore we have $x_8(t)=x_8(0)$ for all $t \geq 0$. This requirement combined with \eqref{parametrization} imposes further conditions on the rate constants which must be satisfied for a positive steady state to exist.

\subsection{MAPK Model Example}
\label{app:mapk}

Reconsider the MAPK model \eqref{mapk}.

\paragraph*{Step 1:} We were able to use the algorithm described in Section \ref{sec:blp} to determine the following structural translation:
\[\begin{tikzcd}
X+K+M \arrow[d,xshift=+0.5ex,"r_{16}"] \arrow[r,  yshift=+0.5ex, "r_1"] & XK+M \arrow[l,yshift=-0.5ex,"r_2"] \arrow[r,"r_3"] & X_p + K +M \arrow[dr,xshift=+0.5ex,"r_{11}"] \arrow[dl,xshift=-0.5ex,"r_{12}"'] \arrow[r,yshift=+0.5ex, "r_4"] & X_pK +M \arrow[l,yshift=-0.5ex,"r_5"] \arrow[r,"r_6"]  & X_{pp}+K +M \arrow[d,xshift=+0.5ex,"r_{7}"]\\
XM + K \arrow[u,xshift=-0.5ex,"r_{15}"] & X^*_pM + K \arrow[ru,yshift=-0.5ex,"r_{13}"'] \arrow[l,"r_{14}"] & & X_pM + K \arrow[ul,yshift=-0.5ex,"r_{10}"]  & X_{pp}M+K \arrow[u,xshift=-0.5ex,"r_8"] \arrow[l,"r_9"]
\end{tikzcd}\]

\paragraph*{Steps 2 \& 3:} The complexes $X+K+M$ and $X_p + K +M$ are both assigned multiple kinetic complexes and therefore must be split. Setting $1$ and $3$ as the distinguished complexes and introducing phantom edges gives the following $V^{\star}$-directed GCRN:
\begin{equation}\label{mess4}\scriptsize
\begin{tikzcd}[column sep=0.5cm]
\mbox{\ovalbox{$\begin{array}{c}  1\star \\ \dag \\ \end{array}\Bigg\lvert \begin{array}{c} X+K+M \\ (X+K) \end{array}$}} \arrow[r,yshift=+0.5ex,"r_1"] \arrow[red,d,"\sigma_1"] & \mbox{\ovalbox{$\begin{array}{c}  2 \\ \\ \end{array}\Bigg\lvert \begin{array}{c} XK+M\\ (XK) \end{array}$}} \arrow[l,yshift=-0.5ex,"r_2"] \arrow[r,"r_3"] 
& \mbox{\ovalbox{$\begin{array}{c}  3\star \\ \ddag \\ \end{array}\Bigg\lvert \begin{array}{c} X_p+K+M\\ (X_p+K) \end{array}$}} \arrow[r,yshift=+0.5ex,"r_4"] \arrow[red,dd,"\sigma_2"]
& \mbox{\ovalbox{$\begin{array}{c}  4 \\  \\ \end{array}\Bigg\lvert \begin{array}{c} X_pK+M\\ (X_pK) \end{array}$}} \arrow[r,"r_6"] \arrow[l,yshift=-0.5ex,"r_5"]
& \mbox{\ovalbox{$\begin{array}{c}  5 \\ \\ \end{array}\Bigg\lvert \begin{array}{c} X_{pp}+K+M\\ (X_{pp}+M) \end{array}$}} \arrow[dd,xshift=+0.5ex,"r_7"] \\
\mbox{\ovalbox{$\begin{array}{c}  11 \\ \dag \\ \end{array}\Bigg\lvert \begin{array}{c} X+K+M\\ (X+M) \end{array}$}} \arrow[d,"r_{16}"] & & & & \\
\mbox{\ovalbox{$\begin{array}{c}  10 \\ \\ \end{array}\Bigg\lvert \begin{array}{c} XM+K\\ (XM) \end{array}$}} \arrow[uu,yshift=0.5ex,bend right=75,"r_{15}"'] & \mbox{\ovalbox{$\begin{array}{c}  9 \\ \\ \end{array}\Bigg\lvert \begin{array}{c} X^*_pM+K\\ (X^*_pM) \end{array}$}} \arrow[ruu,"r_{13}"] \arrow[l,"r_{14}"'] 
& \mbox{\ovalbox{$\begin{array}{c}  8 \\ \ddag \\ \end{array}\Bigg\lvert \begin{array}{c} X_p+K+M\\ (X_p+M) \end{array}$}} \arrow[l,"r_{12}"'] \arrow[r,"r_{11}"]
& \mbox{\ovalbox{$\begin{array}{c}  7 \\ \\ \end{array}\Bigg\lvert \begin{array}{c} X_pM+K\\ (X_pM) \end{array}$}} \arrow[luu,yshift=-0.5ex,"r_{10}"']
& \mbox{\ovalbox{$\begin{array}{c}  6 \\ \\ \end{array}\Bigg\lvert \begin{array}{c} X_{pp}M+K\\ (X_{pp}M) \end{array}$}} \arrow[uu,xshift=-0.5ex,"r_8"] \arrow[l,"r_9"']\\
\end{tikzcd}
\end{equation}
where $\sigma_1$ and $\sigma_2$ indicate the phantom edges, $\dag$ and $\ddag$ indicate equivalence classes of stoichiometrically identical complexes, and $\star$ indicates the distinguished vertex within each class.

\paragraph*{Step 4:} We can compute that the kinetic-order deficiency is one. We therefore have one condition of the form $c^T b = 0$ where $c \in \mbox{ker}(M)$ to satisfy on the rate constants in order to apply the method prescribed by Theorem 14 of \cite{J-M-P}. We suspend discussion of the construction of the matrix $M$ to Step 6, but note that the required condition is
\begin{equation}
\label{condition}
\frac{(k_{11}+k_{12})\sigma_2}{k_{16}\sigma_1} \; \Longrightarrow \; \sigma_2 = \frac{k_{11}+k_{12}}{k_{16}} \sigma_1.
\end{equation}
That is, we eliminate one of our free parameters to satisfy the condition. Since this result is positive, we may proceed.

\paragraph*{Step 5:} After substituting \eqref{condition} into \eqref{mess4}, we can compute the following tree constants
\[\small \begin{aligned}
K_1 & = (k_2+k_3)(k_5+k_6)k_7k_{9}k_{10}(k_{11}+k_{12})k_{12}k_{14}k_{15}\sigma_1\\
K_2 & = k_1(k_5+k_6)k_7k_{9}k_{10}(k_{11}+k_{12})k_{12}k_{14}k_{15}\sigma_1 \\
K_3 & = k_1k_3 (k_5 + k_6) k_7 k_{9} k_{10}(k_{11} + k_{12}) (k_{13} + k_{14})  k_{15} k_{16} \\
K_4 & = k_1k_3k_4k_7k_{9}k_{10}(k_{11}k_{13}+k_{11}k_{14}+k_{12}k_{13}+k_{12}k_{14})k_{15}k_{16}\\
K_5 & = k_1 k_3 k_4 k_6(k_8 + k_{9}) k_{10} (k_{11} k_{13} + k_{11} k_{14} + k_{12} k_{13}+ k_{12} k_{14})k_{15} k_{16}  \\
K_6 & = k_1k_3k_4k_6k_7k_{10}(k_{11}k_{13}+k_{11}k_{14}+k_{12}k_{13}+k_{12}k_{14})k_{15}k_{16}\\
K_7 & = k_1k_3(\sigma_1(k_5+k_6)k_{11}+k_{16}k_4k_6)k_7k_9(k_{11}+k_{12})(k_{13}+k_{14})k_{15}\\
K_8 & =  k_1k_3(k_5 + k_6)k_7 k_{9} k_{10}  (k_{11}+k_{12})(k_{13} + k_{14}) k_{15}  \sigma_1 \\
K_9 & = k_1k_3(k_5+k_6)k_7k_{9}k_{10}(k_{11}+k_{12})k_{12}k_{15}\sigma_1\\
K_{10} & = (k_2\sigma_1+k_3\sigma_1+k_1k_3)(k_5+k_6)k_7k_{9}k_{10}(k_{11}+k_{12})k_{12}k_{14}\sigma_1\\
K_{11} & = (k_2k_5+k_2k_6+k_3k_5+k_3k_6)k_7k_{9}k_{10}(k_{11}+k_{12})k_{12}k_{14}k_{15}\sigma_1^2/k_{16}
\end{aligned}\]

\paragraph*{Step 6:} The log-linear system \eqref{loglinear2} can be set-up for any maximal set of pairs of vertices lying in the same linkage class. We take the pairs
\[\small \{ 1, 2\}, \{1, 3\}, \{1, 4\}, \{1, 5\}, \{1, 6\}, \{1, 7\}, \{1, 8\}, \{1, 9\}, \{1, 10\}, \{1, 11\}.\]
This gives the following log-linear system \eqref{loglinear2}:
\begin{equation}
\label{loglinear3}
\small \left[ \begin{array}{ccccccccccc}-1&0&0&0&-1&1&0&0&0&0&0\\-1&1&0&0&0&0&0&0&0&0&0\\-1&0&0&0&-1&0&1&0&0&0&0\\-1&0&1&1&-1&0&0&0&0&0&0\\-1&0&0&0&-1&0&0&1&0&0&0\\-1&0&0&0&-1&0&0&0&1&0&0\\-1&1&0&1&-1&0&0&0&0&0&0\\-1&0&0&0&-1&0&0&0&0&1&0\\-1&0&0&0&-1&0&0&0&0&0&1\\0&0&0&1&-1&0&0&0&0&0&0\end{array} \right] \left[ \begin{array}{c} \ln(X) \\ \ln(X_p) \\ \ln(X_{pp}) \\ \ln(M) \\ \ln(K) \\ \ln(XK) \\ \ln(X_pK) \\ \ln(X_{pp}M) \\ \ln(X_pM) \\ \ln(X^*_pM) \\ \ln(XM) \end{array} \right] = \left[ \begin{array}{c} \ln\left( K_2/K_1 \right) \\ \ln\left( K_3/K_1 \right) \\ \ln\left( K_4/K_1 \right) \\ \ln\left( K_5/K_1 \right) \\ \ln\left( K_6/K_1 \right) \\ \ln\left( K_7/K_1 \right) \\ \ln\left( K_8/K_1 \right) \\ \ln\left( K_9/K_1 \right) \\ \ln\left( K_{10}/K_1 \right) \\ \ln\left( K_{11}/K_1 \right) \end{array} \right]
\end{equation}
Note that, in Step 4, we used the left kernel vector $c = (0,1,0,0,0,0,-1,0,0,1)$ of the coefficient matrix $M^T$ of \eqref{loglinear3}. Since we have satisfied the condition $c^T b = 0$ with \eqref{condition}, this is a consistent system. We can solve this system and exponentiate to obtain the following steady state parametrization:
\begin{equation}
\label{parametrization2}
\small \begin{aligned}
X & = \frac{(k_2+k_3)k_{15}\tau_1}{(k_2+k_3)\sigma_1+k_1k_3}&X_pK & = \frac{k_1k_3k_4(k_{13}+k_{14})k_{15}k_{16}\tau_1\tau_2}{((k_2+k_3)\sigma_1+k_1k_3)(k_5+k_6)k_{12}k_{14}\sigma_1}\\
X_p & = \frac{k_1k_3(k_{13}+k_{14})k_{15}k_{16}\tau_1}{((k_2+k_3)\sigma_1+k_1k_3)k_{12}k_{14}\sigma_1}& X_{pp}K & = \frac{k_1k_3k_4k_6(k_{13}+k_{14})k_{15}k_{16}\tau_1\tau_2}{((k_2+k_3)\sigma_1+k_1k_3)(k_5+k_6)k_9k_{12}k_{14}\sigma_1}\\
X_{pp} & = \frac{k_1k_3k_4k_6(k_8+k_9)(k_{13}+k_{14})k_{15}k_{16}^2\tau_1}{((k_2+k_3)\sigma_1+k_1k_3)(k_5+k_6)k_7k_9k_{12}k_{14}\sigma_1^2}&X_pM & = \frac{k_1k_3(\sigma_1(k_5+k_6)k_{11}+k_4k_6k_{16})(k_{13}+k_{14})k_{15}\tau_1\tau_2}{((k_2+k_3)\sigma_1+k_1k_3)(k_5+k_6)k_{10}k_{12}k_{14}\sigma_1}\\
M & = \frac{\sigma_1\tau_2}{k_{16}}&X^*_pM & = \frac{k_1k_3k_{15}\tau_1\tau_2}{((k_2+k_3)\sigma_1+k_1k_3)k_{14}}\\
K & = \tau_2& XM & = \tau_1\tau_2\\
XK & = \frac{k_1k_{15}\tau_1\tau_2}{(k_2+k_3)\sigma_1+k_1k_3} & &
\end{aligned}
\end{equation}
in the parameters $\sigma_1, \tau_1, \tau_2 \in \mathbb{R}_{>0}$.

The parametrization \eqref{parametrization2} is quite useful in the context of determining the capacity for mono and multistationarity within stoichiometric compatibility classes of the mass-action system \eqref{eqn:de} corresponding to the MAPK network \eqref{mapk}. The steady states are not toric so that the results of \cite{M-D-S-C} and \cite{M-F-R-C-S-D} cannot be applied. We can, however, apply the computational procedure of Corollary 2 of \cite{C-F-M-W2016}. To satisfy the assumptions, we note that the network has the following conservation laws and is therefore dissipative:
\[\begin{aligned}
X_{tot} & = X + X_p + X_{pp} + XK + X_pK + XM + X_pM + X_p^*M + X_{pp}M\\
K_{tot} & = K + XK + X_pK \\
M_{tot} & = M + XM + X_pM + X_p^*M + X_{pp}M.
\end{aligned}\]
It also has no critical siphons so that there are no boundary equilibria. Computing the function $a(\hat{x})$ of \cite{C-F-M-W2016} evaluated along the parametrization \eqref{parametrization2} yields a rational function in the three parameters $\sigma_1, \tau_1,$ and $\tau_2$ with a strictly positive denominator. It can be checked that, in the numerator of $a(\hat{x})$, $\tau_1^2 \tau_2^2$ and $\sigma_1^2 \tau_1 \tau_2^3$ are extremal with respect to the corresponding Newton polytope and that the coefficients have mixed sign in $a(\hat{x})$. It follows that the mechanism exhibits multistationarity for some choices of rate constants and initial conditions. It should be noted that the parametrization \eqref{parametrization2} reduces the dimension of the system from $11$ variables to $3$ which allows significantly faster computation and analysis of $a(\hat{x})$.

\end{document}